\documentclass[aos,preprint,11pt]{imsart}

\RequirePackage[OT1]{fontenc}
\RequirePackage{amsmath,amsthm,amsfonts,epsfig,amssymb,eucal,eufrak}
\RequirePackage[colorlinks,citecolor=blue,urlcolor=blue]{hyperref}
\usepackage{mathrsfs}
\usepackage{graphicx}

\usepackage[numbers, sort, comma, square]{natbib}
\usepackage{palatino}

\usepackage{etoolbox}
\makeatletter
\patchcmd{\subsection}{\scshape}{\bf}{}{}
\makeatother

\usepackage{geometry}
\newcommand{\lb}{\label}

\newcommand{\nn}{\nonumber}

\newcommand{\sumin}{\sum_{i=1}^{n}}

\newcommand{\be}{\begin{equation}}
\newcommand{\ee}{\end{equation}}
\newcommand{\beqa}{\begin{eqnarray*}}
\newcommand{\eeqa}{\end{eqnarray*}}
\newcommand{\beqn}{\begin{eqnarray}}
\newcommand{\eeqn}{\end{eqnarray}}
\newcommand{\ba}{\begin{array}}
\newcommand{\ea}{\end{array}}
\newcommand{\bc}{\begin{center}}
\newcommand{\ec}{\end{center}}
\newcommand{\btab}{\begin{tabular}}
\newcommand{\etab}{\end{tabular}}

\newcommand{\kn}{K_n}
\newcommand{\kz}{K_0}
\newcommand{\kns}{K_n^*}
\newcommand{\kzs}{K_0^*}
\newcommand{\dmu}{d \mu(\theta)}
\newcommand{\dnu}{d \nu(x)}

\newcommand{\mn}{m_n}
\newcommand{\fiit}{f_{i-1} (\theta)}
\newcommand{\mix}{m_i(x)}

\newcommand{\miixi}{m_{i-1}(X_i)}
\newcommand{\mxi}{m(X_i)}
\newcommand{\miix}{m_{i-1}(x)}

\newcommand{\Fii}{\mathcal{F}_{i-1}}

\newcommand{\pxit}{p(X_i | \theta)}
\newcommand{\pxt}{p(x | \theta)}

\DeclareMathOperator{\E}{E}

\newtheorem{rmrk}{\indent\sc Remark}
\startlocaldefs
\numberwithin{equation}{section}
\theoremstyle{plain}
\newtheorem{thm}{Theorem}[section]

\newtheorem{lemma}{\indent\sc Lemma}
\newtheorem{proposition}{\indent\sc Proposition}

\endlocaldefs

\begin{document}

\begin{frontmatter}
\title{Stochastic approximation algorithm for estimating mixing distribution for dependent observations}
%\thankstext{T1}{Footnote to the title with the ``thankstext'' command.}

\begin{aug}
\author{\fnms{Nilabja } \snm{Guha}\thanksref{m1}\ead[label=e1]{nilabja$\_$guha@uml.edu}},
\author{\fnms{Anindya} \snm{Roy}\thanksref{m2}\ead[label=e2]{anindya@umbc.edu}}
%\end{aug}
%\and
%\author{\fnms{Third} \snm{Author}\thanksref{t1,m2}
%\ead[label=e3]{third@somewhere.com}
%\ead[label=u1,url]{http://www.foo.com}}
%
%\thankstext{t1}{Some comment}
%\thankstext{t2}{First supporter of the project}
%\thankstext{t3}{Second supporter of the project}
\runauthor{Guha and Roy}

\affiliation{Deaprtment of Mathematical Sciences, University of Massachusetts Lowell \thanksmark{m1};   Department of Mathematics and Statistics, University of Maryland Baltimore County \thanksmark{m2}}
%
%\address{Address of the First and Second authors\\
%Usually a few lines long\\
%\printead{e1}\\
%\phantom{E-mail:\ }\printead*{e2}}
%
%\address{Address of the Third author\\
%Usually a few lines long\\
%Usually a few lines long\\
%\printead{e3}\\
%\printead{u1}}
\end{aug}

\begin{abstract}
Estimating the mixing density of a mixture distribution remains an interesting problem in statistics.  Using a stochastic approximation  method,  \citet{newton1999recursive} introduced a  fast recursive algorithm for  estimating  the  mixing  density of a mixture.   Under suitably chosen weights the stochastic approximation estimator converges to the true solution.    In \citet{tokdar2009consistency} the consistency of this recursive estimation method was established. However the current results  of consistency of the resulting recursive estimator use independence among observations as an assumption. We extend the investigation of performance of Newton's algorithm  to several dependent scenarios. We  prove that the original algorithm under certain conditions remains consistent even when the observations are arising from a weakly dependent stationary process with the target mixture as the marginal density.  We show consistency under a decay condition on the dependence among observations when the dependence is characterized by a quantity similar to mutual information between the observations.\end{abstract}

\end{frontmatter}
\thispagestyle{plain} % remove running headers

\section{Introduction}
Stochastic approximation (SA) algorithms which are stochastic optimization techniques based on recursive update, have many applications in optimization problems arising in different fields such as  engineering and  machine learning.  A classical and pioneering example of an SA can be found in  \citet{robbins1951stochastic} where a recursive method is introduced for finding the root of a function. In particular, suppose we have a non-increasing   $h$, where values of $h$ are observed with errors, i.e., we observe the value of $h$ at a point $x_n$, as $y_n=h(x_n)+{\epsilon}_n$.  Here ${\epsilon}_n$'s are i.i.d errors with mean zero and finite variance. The Robins--Monro stochastic algorithm recursively approximate the solution of $h(x)=\alpha_0$ as 
\be
x_{n+1}=x_n+w_n(y_n-\alpha_0),
\lb{SAC}
\ee
where $w_n$  is a sequence of weights satisfying $w_n>0, {\sum}_n w_n=\infty \text{ and }$ $ {\sum}_n {w_n}^2<\infty$.
Under  \eqref{SAC}, the sequence $\{x_n\}$ approaches  the true root $x_0$, such that $h(x_0)=\alpha_0$. Subsequent developments include rate of convergence, optimum step-size, convergence under convexity etc (see  \citet{chung1954stochastic}, \citet{fabian1968asymptotic},  \citet{polyak1992acceleration}  ).  Also see \citet{sharia2014truncated},\citet{kushner2010stochastic}   for recent developments such as adaptively truncating the solution to some domain. More generally, for a predictable sequence $R_t(Z_t)$ with $R_t(z_0)=0; z_0\subset \mathbb{R}^m$ the recursive solution is a  limit to the recursion
\begin{eqnarray}
Z_t=Z_{t-1}+\gamma_t(Z_{t-1})[R_t(Z_{(t-1)})+\epsilon_t(Z_t)]; t=1,2,\dots
\lb{rm_multi}
\end{eqnarray}
where $Z_0\subset \mathbb{R}^m$ is the initial value, $\gamma_t(Z_{t-1})$ possibly state dependent matrix of step sizes, and $\epsilon_t(Z_t)$ is a mean zero random process.
The idea of stochastic approximation was cleverly used in a predictive recursion (PR) algorithm in \citet{newton1999recursive}   for  finding the mixing component of a mixture distribution.  

Mixture models are popular statistical models that provide a nice compromise between the flexibility offered by nonparametric approaches and the efficiency of  parametric  approaches. Mixture models are increasingly used in messy data situations where adequate fit is not obtained using conventional parametric models. Many algorithms,  mostly variants of expectation maximization  (EM) algorithm and Markov chain Monte Carlo (MCMC) algorithms, are currently available for fitting mixture models to the data. Specifically, these algorithms would fit a marginal model of the form
\be 
{m}_f =\int_{\Theta}p(x|\theta)dF(\theta) .
\lb{marginal}
\ee 
to the data $X_1, \ldots,X_n$ assuming a form of the mixing kernel $p(x | \theta).$  A related problem is recovering (estimating) the mixing distribution $F$. The problem of estimating the density $f$ of $F$ (with respect to some dominating measure $\mu$) in a nonparametric set up  can be a challenging exercise and can be numerically taxing when full likelihood procedures (such as nonparametric MLE or nonparametric Bayesian) are used. However the full likelihood procedures generally enjoy desirable large sample properties, such as consistency. 

Let $X\in \chi$ be a random variable with distribution $p(x|\theta)$, where $ \theta \in \Theta $ is the latent variable and $f(\theta)$ be the mixing density. Let $\nu$ and $\mu$ be the sigma-finite measures associated with $\chi$ and $\Theta$. 
The recursive estimation algorithm  \citet{newton1999recursive}  for estimating $f(\theta)$ then starts with some initial $f_0(\theta)$ which has the same support as the true mixing density $f(\theta)$.
The update with each new observation $X_i$ is given by
\be 
f_i(\theta)=(1-w_i)f_{i-1}(\theta)+w_i \frac{p(X_i|\theta)f_{i-1}(\theta)}{\miixi} 
\lb{newton}
\ee 	
where $\miix = \int_{\Theta}p(x|\theta) \fiit d\mu(\theta)$ is the marginal density of $X$ at the $i$th iteration based on the mixing density $f_{i-1}$ obtained at the previous iteration. 

From equation \eqref{newton} calculating the marginals on both sides, we have the associated iterations for the marginal densities as 
\be
\mix = \miix\big\{1+w_i(\frac{\int \pxt \pxit \fiit  \dmu }{\miix\miixi}-1)\big\}.
\lb{ntnmrg}
\ee
One way to connect the PR update to SA algorithm is to  minimize the Kullback--Leibler (KL) distance between the proposed marginal and the true marginal and considering the minimizer of the corresponding Lagrange multiplier (\citet{ghosh2006convergence}; \citet{martin2008stochastic}).
The support of $f_0$ can be misspecified. In such cases PR estimate can be shown to concentrate around a density in the space of proposed mixture densities that is nearest in KL distance to the true marginal (\citet{martin2009asymptotic}).  Similar results can be found  for general posterior consistency such as developments in \citet{kleijn2006misspecification} where posterior consistency under misspecification was shown under conditions like convexity.

The PR algorithm is computationally much faster than the full likelihood methods.  \citet{ghosh2006convergence},  and later \citet{tokdar2009consistency}   established the consistency of Newton's predictive recursion algorithm, thereby putting the PR algorithm  on solid theoretical footing. Some subsequent developments on the rate of convergence  can be found in \citet{martin2008stochastic}, \citet{martin2009asymptotic}. Later developments focus on application on semiparametric model, predictive distribution calculation in \citet{hahn2018recursive}, \citet{martin2016semiparametric}.  Like the proof of original SA algorithm the proving the consistency for the recursive estimate of the mixture distribution depends on a Martingale difference sum construction.
Because the algorithm depends on the order of $X_i$'s, the resulting estimator is not a function of sufficient statistics, the consistency of the PR solution cannot be drawn from the consistency of frequentist and Bayesian  method (e.g using DP mixture model; see \citet{ghosal1999posterior}, \citet{lijoi2005consistency} ).  In \citet{ghosh2006convergence} and \citet{tokdar2009consistency} the  martingale based method from \citet{robbins1951stochastic} was adapted in density estimation setting in a novel way to  show the almost sure convergence of the estimator in the weak topology and Kullback--Leibler (KL) divergence. The PR solution was shown to be the Kullback--Leibler divergence minimizer between the true marginal and the set of proposed marginals. 

One of the key assumptions in the existing proof of consistency is that the observations are independent. This assumptions significantly limits the scope of application for the PR algorithm, where some naturally occurring dependence may be present, for example cases of mixture of Markov processes. The main result of this paper is that  the predictive recursion continues to provide consistent solution even under weakly dependent stationary processes
as long as the dependence decays reasonably fast. This vanishing dependence can be connected with and can be quantified by  the information theoretic quantities such as  mutual information between the marginal and the conditional densities of the process. We use a novel sub-sequence argument to tackle the dependence among the observations and prove consistency of the PR algorithm when such vanishing dependence is present.  At the same time we derive a bound for the convergence rate for the PR algorithm under such dependence. 
As a special case, we later consider the example of general $M$ dependent cases, where the consistency of the recursive estimator holds under weaker conditions. 
%Finally, for non-stationary processes where the conditional distribution of $X_i$ given lags $X_{i-1}, \ldots, X_{i-k}$ is a mixture with a  kernel that depends on $i$ and where the existing PR algorithm is no longer applicable, we propose a modified algorithm and show that it consistently estimates the mixing density. 
In all the cases we also  investigate  convergence under misspecification of the support of the mixing density and concentration around the KL projection to the misspecified model.

The arrangement of this article is the following. In  Section 2, we provide the background and basic framework for the  martingale based argument. Section 3 presents the main results regarding convergence and the rates  for the  weakly dependent cases, and address the special case of mixture of Markov processes.  In Section 4, we consider the special case of $M$- dependent sequences.
%Section 5 contains results for a modified kernel based algorithm that addresses nonstationarity of the observed process. 

\section{Preliminaries and revisiting the independent case} 
Our notation and initial framework will  follow the original martingale type argument of \citet{robbins1951stochastic} and the developments for the independent case  in the literature, especially in   \citet{ghosh2006convergence} and \citet{tokdar2009consistency}. Thus,  we first introduce the notation and revisit the main techniques used in the proof of consistency of the PR estimator in  \citet{tokdar2009consistency}. The discussion will also illustrate the need for generalization of the techniques to the dependent case. 

  A recursive formulation using  KL divergence  along with   martingale based decomposition was used in  \citet{tokdar2009consistency} who established  convergence of the recursive estimators for the mixing density and that of the marginal density. Specifically, if $\kn = \int f\, log(f/f_n) \dmu$ and  $\mathcal{F}_i = \sigma(X_1, \ldots, X_i)$ is the $\sigma-$ algebra generated by the first $i$ observations, then the following recursion can be established: 
\be 
\kn - \kz=\sumin w_i V_i- \sumin w_i M_i + \sumin E_i 
\lb{kl1}
\ee 
where 
\begin{eqnarray*}
M_i&=& -\E \Big{[}1 - \frac{\mxi}{\miixi}|\Fii\Big{]}; m(X_i)=\int p(X_i|\theta)f(\theta)d\mu (\theta),\\
V_i&=&(1-\frac{\mxi}{\miixi}) + M_i, \\
E_i&=& \int R(X_i, \theta) f(\theta) \dmu, \\
R(X_i, \theta)& =& w_i^2(\frac{\pxit}{m_{i-1}(X_i)}-1)^2R(w_i(\frac{\pxit}{m_{i-1}(X_i)}-1)) 
\label{KLdiff}
\end{eqnarray*}
and $R(x)$ is defined through the relation $\log (1 + x) = x - x^2 R(x)$ for $x > -1.$ The remainder term $R$ satisfies  $0\leq R(x)\leq\text{ max}\{1,(\frac{1}{1+x})^2\}/2$. 
The corresponding similar recursion for the KL divergence of the marginal densities, $\kns = \int m\, log(m/\mn) \dnu$, is then
\be
\kns - \kzs =\sumin  w_i V^*_i - \sumin w_i M^*_i + \sumin  E^*_i
\lb{kl2}
\ee 
where
\begin{eqnarray*}
g_{i,x}(\theta)&=&\frac{\pxt \fiit}{\miix},\\
h_{i,x'}(x)&=&\int g_{i,x'}(\theta)\pxt \dmu,\\
R^*(X_i,x)&=&w_i^2(\frac{h_{i,X_i}(x)}{\miix}-1)^2R(w_i[\frac{h_{i,X_i}(x)}{\miix}-1]),\\
V_i^*&=&(1-\int\frac {h_{i,X_i}(x)}{\miix} m(x) \dnu + M_i^*,\\
M_i^*&=&\int_{\chi} \int _{\chi}\frac{{h_{i,x'}(x)}}{m_{i-1}(x)}m(x) m(x^{\prime}) \dnu d \nu(x^{\prime}) - 1,\\
 E_i^*&=&\int_{\chi} R(X_i,x)m(x)\dnu.
\end{eqnarray*}
It is assumed that  $K_0$ and $K_0^*$, corresponding to initial starting point $f_0$, are finite. The main idea in  \citet{tokdar2009consistency} was to recognize that $\sumin w_i V_i$ , $\sumin w_i V_i^*$ are mean zero square integrable  martingales and hence almost surely convergent, $\sumin w_i M_i$, $\sumin w_i M_i^*$ are positive, and $\sumin E_i$, $\sumin E_i^*$ have  finite limits almost surely. Putting these facts together  \citet{tokdar2009consistency} established that $\kns$ necessarily converges to zero limit almost surely, thereby establishing consistency of the predictive recursion sequences $f_n$ and $m_n$ in weak topology.  Using developments in  \citet{robbins1971convergence} it can be argued that the PR solution converges to the KL minimizer if the support is misspecified.

Hence, the theoretical results regarding convergence of such PR algorithms in current literature relies  on independence of $X_i$ and also assumed that the true marginals at each $i$ were  the same and equal to $m(x)$. There does not seem any obvious way of extending the proof to the case when the observations are not independent.

 We propose  a proof for convergence of the predictive recursion algorithm in the case when the observations are dependent. While the proof will use  martingale construction technique similar to that of  \citet{tokdar2009consistency}, there are significant differences in the approach that allows us to address the case of dependent observations, and tools required to address the dependence will be detailed in next section. 
 %We also  address the case when the observations are dependent with varying marginal densities. This requires modifying the original PR algorithm to address different marginal densities for the observations. 
 The main contributions of the paper are summarized in the following: 
\begin{enumerate}
\item
 We  show that the PR algorithm continues to be consistent under weakly dependent  processes where the marginal stationary distribution at each $i$ is a mixture with respect to a fixed kernel. The consistency is obtained under additional conditions on the kernel and the parameter space.
If the dependence decays exponentially, under additional conditions, the original PR estimate is shown to be consistent. The decaying dependence is characterized by a special case of expected $f$ divergence between marginal and conditional densities. 
\item  We establish the   convergence rate when the support is correctly specified and or misspecified with decaying dependence.
\item We show the result for finite mixture of Markov processes and  general $M$-dependent processes under milder conditions.
%\item We suggest a modified recursive algorithm when the conditional distribution of the current observation given past observations is a  mixture of Markovian kernels (possibly depending on the index of the observation)   and establish consistency of such an updating scheme. 
\end{enumerate}
The next section describes the main results as well the  tools needed for the proof of consistency in the dependent case that uses a novel sub-indexing argument.

\section{Main results}
%Some of the desirable properties of the PR recursion continue to hold under some mild forms of dependence as well. If the process is a stationary process whose marginal is a mixture and the serial dependence dies  off relatively fast, then  it is still possible to consistently estimate the true mixing density using the predictive recursion algorithm. The ideas are formalized in the following. 

To establish consistency of the PR algorithm when observations are dependent, we will need to control the 
expectation of  ratios of $k $-fold products ($(k\geq 1)$) of densities  at two different parameter values. Suppose the parameter  $\theta$ lies in a compact subset $\Theta$ of an Euclidean space. Let $\widehat{\Theta}$ denote a closed convex set containing $\Theta$ and assume 
  $p(x|\theta)$ is well defined for $\theta \in \widehat{\Theta}$. Assume there is  a finite set $\Theta_H = \{\theta_j,j=1\dots n_H\} \in \widehat{\Theta},$ (typically will be the extreme points when $\widehat{\Theta}$ is a convex polytope) and a compact set  $\chi_c \subset \chi$ such that the following holds.

\begin{itemize}

\item[]C1.  For any $x \notin \chi_c$   there exists $\theta_1^x,\theta_2^x \in \Theta_H$ such that $\text{sup}_{\theta,\theta'\in \Theta}\{ \frac{p(x | \theta)}{p(x|\theta')}\}\leq c_u\frac{p(x|\theta_2^x)}{p(x | \theta_1^x)}$, for some  $c_u>1$.  This condition  is satisfied  with $c_u=1$  for $\theta^x_l=\underset{\theta} \arg\inf p(x | \theta) \in \Theta_H $ and $\theta^x_u=\underset{\theta} \arg\sup p(x | \theta)\in \Theta_H$.
\vskip 5pt

\item[]C2. There exists $a > 0$ such that $\underset{x\in \chi_c,  \theta \in \Theta} \inf p(x|\theta)>a$.
\vskip 5pt

\item[]C3. There exists $b < \infty$ such that $\underset{x\in \chi_c,  \theta \in \Theta} \sup p(x|\theta)<b$.

\end{itemize}

Without loss of generality, we can assume that $0<a<1<b$. 
Under C1---C3, the ratio of the marginals  can be bounded by the ratios of conditionals on the finite set $\Theta_H$ and a function of $a$ and $b$. For many  common problems,  such as  location mixtures of $p(x|\theta)$, a finite set $\Theta_H$ exists that satisfies the assumption. This condition can be seen as a generalization of Monotone Likelihood Ratio  property in higher dimensions.
The  condition is satisfied  in a  general multivariate normal mean mixture with known covariance matrix, where the mean parameter is constrained to a set $\Theta$  in  $\mathcal{R}^d$,  contained in a closed large convex polytope ${\widehat{\Theta}}$, and the set   $\Theta_H$ then consists of suitable selected points on the boundary of  $\widehat{\Theta}$ and $\chi_c=\widehat{\Theta}$.
% For classes such as  scale ($\theta$) mixture of Gamma with fixed shape parameter, Poisson mean ($\theta$) mixture this condition is  also satisfied when $\Theta$ is compact and does not contain zero. 
\begin{proposition}
Under C1--C3 , for any two distribution $f_1$ and $f_2$ on compact $\Theta$, the ratio of the marginals under mixing densities  $f_1$ and $f_2$ can be bounded as \begin{equation} 
\frac{m_{f_1}(X_i)}{m_{f_2}(X_i)}\leq A_1(X_i)=c_u \sum_{\theta_{k},\theta_l \in \Theta_H}\frac{p(X_i|\theta_k)b}{p(X_i|\theta_l)a}.
\label{chaul}
 \end{equation}
 \label{chaul_bd}
\end{proposition}
\begin{proof}
Given in the Appendix. 
\end{proof}
\begin{rmrk}
For finite $\Theta$ the result in Proposition \ref{chaul_bd} holds trivially  using $\Theta$ in place of $\Theta_H$.
\end{rmrk}

When $X_i$'s have a fixed  marginal density $m$,  PR algorithm can still be applied in spite of dependence  if the dependence decreases rapidly  with distance along the sequential order in which  $X_i$'s enter the algorithm. We assume the following $\alpha$-mixing type condition for the $X_i$  sequence. The decaying dependence is expressed as  a special case of expected $f$ divergence  (expected $\chi^2$-distance) between the  marginal and the conditional densities.  

Let $X_{1:i}=\{X_1,\ldots,X_i\}$ and let $m(X_{i+n} |X_{1:i})$ be the conditional density/pmf of $X_{i+n}$ given $X_{1:i}$. We assume, 

\begin{eqnarray}
\hspace{0.2in}\text{sup}_i E\big[\int (\frac{m(X_{i+n} |X_{1:i})}{m(X_{i+n})}-1)^2m(X_{i+n})d(.)\big] \leq c_0^2\rho^{2n}
\lb{strdep}
\end{eqnarray}
where $c_0>0$ and $0 < \rho<1$. 

Let ${\bf H}(X_{i+n},X_{1:i})=\int  \log\frac{m(X_{i+n}, X_{1:i})}{m(X_{i+n})m(X_{1:i})}  m(X_{i+n},X_{1:i})d(\cdot)$ denote the {\it mutual information} between $X_{i+n}$ and $X_{1:i}$ and let ${\bf H}_a(X_{i+n},X_{1:i})=$ $\int  |\log\frac{m(X_{i+n}, X_{1:i})}{m(X_{i+n})m(X_{1:i})}|$  $ m(X_{i+n},X_{1:i})d(\cdot)$. 
The condition given in \eqref{strdep} implies exponential decay of mutual information ${\bf H}(X_{i+n},X_{1:i})$. When the conditional densities are uniformly  bounded and uniformly bounded away from zero, \eqref{strdep} is satisfied if   ${\bf H}_a(X_{i+n},X_{1:i})$ decays exponentially. This result can be  summarized in the following proposition.
\begin{proposition}
The condition \eqref{strdep} implies $\text{sup}_i {\bf H}(X_{i+n},X_{1:i})\leq  c_0^2\rho^{2n}$ where $c_0,\rho$ are given in \eqref{strdep}. If $\text{inf}_{i,n}p(X_{i+n}|X_{1:i})>0$ and $\text{sup}_{i,n}p(X_{i+n}|X_{1:i})<\infty$  then condition \eqref{strdep} holds if  $\text{sup}_i{\bf H}_a(X_{i+n},X_{1:i})\leq c_1\rho^{2n}$, for some $c_1>0$. 
\label{infm_bound}
\end{proposition}
\begin{proof}
The result follows from the relationship between $f$-divergence and mutual information and is omitted.
\end{proof}

In addition, we assume the following conditions which are similar to those in  
\citet{tokdar2009consistency}:

\begin{itemize}

\item[B1] $w_i \downarrow 0$, and $w_i \sim i^{-\alpha}$,  $\alpha \in (0.5,1]$.
\item[B2] For  $\theta_{l_1},\theta_{l_2}$'s $ \in \widehat{\Theta}$ for $E[(\prod_{l=1}^{n_1}\frac{p(X_{j_l}|\theta_{l_1})}{p(X_{j_l}|\theta_{l_2})})^2]\leq b_0^{2n_1}\prod_{l=1}^{n_1}E[(\frac{p(X_{j_l}|\theta_{l_1})}{p(X_{j_l}|\theta_{l_2})})^2]$, for some $b_0>0$ and where $j_l 's \in\{1,\dots n\} $ and are distinct,  and $n_1\leq n$. Assume, $b_0\geq 1$ without loss of generality. 
\item[B3]$sup_{\theta_1,\theta_2,\theta_3 \in \widehat{\Theta}} E_{\theta_3}[( \frac{p(X_i|\theta_{1})}{p(X_i|\theta_{2})})^2]<B<\infty$
\item[B4] The map $\theta \to p(x|\theta)$ is bounded and continuous for $ x \in \chi $.

\end{itemize}

Condition $[\rm{B2}]$ is needed for $n_1$ fold products over different indices for the dependent case. This condition will later be verified for some of the examples considered. Condition $[\rm{B2}]$ can be omitted if stricter moment condition $[\rm{B2}']$ is assumed, which is given later.  We can now state our main results. Theorem~\ref{wdependent}  shows consistency when the support of the mixing density is finite while  Theorem~\ref{wdependent2} establishes consistency for general support under slightly more restrictive conditions.

\begin{thm} Let $\{X_i\}$ be a sequence of random variables satisfying \eqref{strdep} where $X_i$ has a fixed marginal density $m(\cdot)$ given by $m(x) = \int_{\Theta} p(x | \theta) f(\theta) d\theta$ and the support of $f$, $\Theta$, is a finite set.
Assume that the initial estimate $f_0$ in \eqref{newton} has the same support as $f$.  Then under \mbox{\rm{B1--B4}}, \mbox{\rm{C1--C3}}, the estimator  $f_n$ in  \eqref{newton} converges to the true mixing density $f$ with probability one, as  the number of observations $n$ goes to infinity. 
\lb{wdependent}
\end{thm}

\begin{proof}
The terms in the decomposition of the KL divergence, \eqref{KLdiff} can no longer be handled in the manner as in the independent case. We use further {\it sub-indexing} of the terms to obtain appropriate convergence results under the dependence condition \eqref{strdep}. 

We first partition the positive natural numbers into sub-sequences. 
Let  $(j)_r$ denote the  $r$th term in $j$th sub-sequence  and let  $\Psi(r,j)$ denote its value. Then the sub-sequences are constructed in the following manner:
\begin{enumerate}
\item
Let  $\Psi(1,1)  = 1$.
\item
Let $\Psi(1,j) = \inf\{\Psi(r,j-1): \Psi(r, j-1) > j^K\} + 1$ for some fixed positive integer $K > 1$. 
\item
The terms in the $(j-1)$th sub-sequence are immediately followed by the next available term in the $j$th sub-sequence unless there are no such terms in which case it is followed by the next term in the first sub-sequence.
\end{enumerate}
By construction,  $j^K\leq (j)_1\leq (j+1)^K$. 
For  convenience of notation, we denote by $i_{last}$ the integer $\Psi(r-1,j)$ whenever $i = \Psi(r,j)$ and $r>1$. We define $i_{last}=i-j$ if $i=\psi(1,j)$. Let $F_{last,i}$ denote the $\sigma$-algebra generated by the collection $\{X_1, \ldots, X_{i_{last}}\}$. 
Similarly, Let, $p_{i,last}$ denote the conditional density of $X_{i}$ given $X_1,\dots, X_{i_{last}}$.
Let  $L_n$ be the number of  subsequences constructed for $i=n$, that is we consider $j$ such that $\Psi(r,j) \leq n$ for some $r\geq 1$. Clearly, $L_n\leq n$. Also,  let $N_{j,n}$ be the number of terms in the $j$th sequence upto $i=n$, such that $\Psi(r,j) \leq n.$ 
As an example, for  $K=2$ the following pattern arises for the sub-sequences\\
\begin{eqnarray*}
\overbrace{(1)_1}^{\tilde{\bf B}_1},\overbrace{(1)_2}^{\tilde{\bf B}_2},\overbrace{(1)_3}^{\tilde{\bf B}_3},\overbrace{(1)_4,(2)_1}^{\tilde{\bf B}_4},\overbrace{ (1)_5,(2)_2}^{\tilde{\bf B}_5},\overbrace{ (1)_6,(2)_3,(3)_1}^{\tilde{\bf B}_6}\dots,\overbrace{(1)_{i_1},(2)_{i_2},\dots, (i+1)_1}^{\tilde{\bf B}_{f(i)}}
\lb{martcon}
\end{eqnarray*}
Thus,  $\Psi(1,2)=5$, $\Psi(1,3) = 10$ $N_{1,10} = 6$ and $L_{10} = 3$ in this example.

%The main idea would be rearrange the sums in \eqref{KLdiff} to have  martingale difference quantities which converge almost surely and 
%The role of $K$ is that by choosing $K$ suitably, we can show with probability one finitely many martingale sums are relevant to  our calculations.

From equation \eqref{newton}, similar to equation \eqref{KLdiff}  we have
\begin{eqnarray}
K_n-K_0=\sum_{j=1}^{{L_n}} S_{v,j}^{(n)}-\sum_{j=1}^{L_n} S_{M,j}^{(n)}+\sum_{i=1}^nS^{\Delta}_i+\sum_{i=1}^n E_i .
\lb{wkkl1}
\end{eqnarray}
where
\begin{eqnarray*}
S_{v,j}^{(n)} &=&\sum_{r=1}^{N_{j,n}}v_{\Psi(r,j)};\\
S_{M,j}^{(n)}&=&\sum_{r=1}^{N_{j,n}}w_{\Psi(r,j)}M_{\Psi(r,j)};\\
v_{i}&=&w_{i}\Big( (1-\frac{m(X_i)}{m_{i-1}(X_{i})}) - E\big[(1-\frac{m(X_{i})}{m_{i-1}(X_{i})}|F_{last,{i}}\big]\Big); \\
M_i&=&E\big[\frac{m(X_i)}{m_{i_{last}}(X_i)} - 1 |F_{last,i}\big];\\
S^{\Delta}_{i}&=&w_{i}E\big[\frac{m(X_i)}{m_{i-1}(X_i)}(\frac{m_{i-1}(X_{i})}{m_{{i}_{last}}(X_{i})} - 1)|F_{last,i})\big],
\end{eqnarray*}
and $E_i$ is as defined following \eqref{KLdiff}.  Next we show convergence of the different parts.
%At first we show the convergence of  $\sum_{j=1}^{L_n} S^{(n)}_{v,j}$, $\sum_{i=1}^n E_i$.
%
%\vspace{0.2in}

\noindent \underline{\it Convergence of $\sum_{j=1}^{L_n} S^{(n)}_{v,j}:$}
We have $E[(S^{(n)}_{v,j})^2]\leq 2\sum w_i^2 E[1+(A_1(X_i))^2]<b'_0$ for some $b'_0>0$. This implies that $S^{(n)}_{v,j}$ is mean zero  a squared integrable martingale with filtration $\sigma(X_1,\cdots,X_{\Psi(N_{j,n},j)})$, and therefore converges almost surely \citep{durrett2019probability}. Thus, outside a set of probability zero $S^{(n)}_{v,j}$'s converges for all $j$, as $j$ varies over a countable set. 

Next, we show that only finitely many martingale sequences in the  $\sum_{j=1}^{L_n} S^{(n)}_{v,j}$ will make  significant contribution with probability one for large $n$. Choose $s>1$ and choose $K$ large enough  such that $ K  > 2s/(2\alpha-1)$. 
From Lemma  \ref{sup_lem}, 
\begin{eqnarray*}
P(\sup_n | S_{v,j}^{(n)}|>\frac{\epsilon}{j^s} \text{ infinitely often})\leq  lim_{n_0 \uparrow\infty}\sum_{j=n_0}^\infty P(\sup_n | S_{v,j}^{(n)}|>\frac{\epsilon}{j^s})\\
\leq \epsilon^{-2}  lim_{n_0 \uparrow\infty}\sum_{j=n_0}^\infty C'_0j^{-K(2\alpha-1)+2s-1}\rightarrow 0
\end{eqnarray*}
 where $C'_0>0$ is a constant.

Hence, outside a null set, say $\Omega_N$ (possibly depending on $\epsilon$),  we have $\sup |S_{v,j}^{(n)}|<\frac{\epsilon}{j^s} $ for all but finitely many $j$'s and $S^{(n)}_{v,j}$'s converge. Fix $\omega\in \Omega\backslash  \Omega_N$,  where $\Omega$ is the underlying product-probability space.   For any $\epsilon_1>0$, we can choose $n_1$ (possibly depending upon $\omega$), such that for $j > n_1$,  $sup | S_{v,j}^{(n)}|<\frac{\epsilon}{j^s}$ and  $\sum_{j=n_1}^\infty  | S_{v,j}^{(n)}|<\epsilon_1$. 
Let  $n_2>\psi(1,n_1)$ large enough such that for  $j\leq n_1$ we have $\sum_{j=1}^{n_1}  | S_{v,j}^{(n_3)} - S_{v,j}^{(n_2)}|<\epsilon_1$ whenever $n_3>n_2$. 
Finally,  \begin{eqnarray*}
|\sum_{j=1}^{L_{n_3}} S^{(n_3)}_{v,j}-\sum_{j=1}^{L_{n_2}}S^{(n_2)}_{v,j}|\leq \sum_{j=1}^{n_1}  | S_{v,j}^{(n_3)}-S_{v,j}^{(n_2)}|+|\sum_{j>n_1} (S^{(n_3)}_{v,j}-S^{(n_2)}_{v,j})|\\
\leq \epsilon_1+\sum_{j>n_1} sup_{n\leq n_2}| S^{(n)}_{v,j}|+\sum_{j>n_1} sup_{n\leq n_3}| S^{(n)}_{v,j}|\leq 3\epsilon_1.
\end{eqnarray*}
As $\epsilon_1>0$ is arbitrary and we can choose $\epsilon_1$ going to zero over a sequence. This implies  $\sum_{j=1}^{L_n} S^{(n)}_{v,j}$ is a  Cauchy sequence with probability one and therefore converges with probability one.

\noindent \underline{\it Convergence of $\sum_{i=1}^n E_i$:}
Using the expression for $A_1(x)$, we have, \[ E[\sum_{i > i_0} |E_i|]\leq\sum_{i > i_0} 4w_i^2E[(1+A_1(x_i))^2]<\infty\] as $w_i(\frac{\pxit}{m_{i-1}(X_i)}-1)>-w_i>-\frac{1}{2}$ for $i > i_0 $  for some $i_0 > 1$, as $w_i\downarrow 0$. Hence, $\sum_{i>i_0}E_i$ converges almost surely from Proposition \ref{pos_sum}. Hence,  $\sum E_i$ converges  almost surely.

\noindent \underline{\it Decomposing ${M}_i$:} 
 We have from  \eqref{wkkl1},
\begin{eqnarray}
M_i&=& \int(\frac{m(x)}{m_{i_{last}}(x)}-1)m(x)\nu(dx)+\int(\frac{m(x)}{m_{i_{last}}(x)}-1)(\frac{p_{i,last}(x)}{m(x)}-1)m(x)\nu(dx) \nonumber \\
&\geq & K^*_{i_{last}}+\int(\frac{m(x)}{m_{i_{last}}(x)}-1)(\frac{p_{i,last}(x)}{m(x)}-1)m(x)\nu(dx).\nonumber\\
\label{wdep_decmp}
\end{eqnarray}
By Cauchy-Schwarz and using [B3] to bound the expression for $A_1(\cdot)$  in \eqref{chaul}, we have 
\begin{eqnarray*}
\delta_{i}=|\int(\frac{m(x)}{m_{i_{last}}(x)}-1)(\frac{p_{i,last}(x)}{m(x)}-1)m(x)\nu(dx) |\leq&\\
\sqrt {\int2(A_1(x)^2+1)m(x)\nu(dx)} &\sqrt{\int (\frac{p_{i,last}(x)}{m(x)}-1)^2m(x)\nu(dx)} .
\end{eqnarray*}
Let $\delta_{i} =$ $ \sqrt{\int (\frac{p_{i,last}(x)}{m(x)}-1)^2m(x)\nu(dx) }$. If $\delta_{i}^{(i)}=E[\delta_{i}]$, then by condition \eqref{strdep} and Jensen's inequality, we have $\delta_i^{(i)} \leq c_0\rho^{(i-i_{last})}$.   By construction of the sequences, a gap $(i - i_{last})$  is equal to  $l$  at most  $(l+1)^K$ times. Also by \eqref{chaul},  $\sqrt {\int2(A_1(x)^2+1)m(x)\nu(dx)}\leq b_1$ for some $b_1>0$.  Thus, $ \sum \delta_{i}^{(i)} \leq c_0b_1\sum_ll^K\rho^{l-1}<\infty$.
Hence, $\sum \delta_{i}$ converges absolutely with probability one and hence, converges with probability one.

\noindent \underline{\it Convergence of $\sum_{i=1}^nS^{\Delta}_i$:} 
By Proposition~\ref{pos_sum},  it is enough to show $\sum E[ |E[w_i(\frac{m(x)}{m_{i-1}(x)}-\frac{m(x)}{m_{i_{last}}(x)})|F_{last,i}]|]\leq \sum E[w_i|\frac{m(x)}{m_{i-1}(x)}-\frac{m(x)}{m_{i_{last}}(x)}|]$ converges. 

%For  index corresponding to $w_i$, the $i_{last}$ and $i$ are approximately $O({i}^{1/K})$ apart. Then writing $(\frac{m_{i-1}(x)}{m_{i_{last}}(x)}-1)$ as the $i-1-i_{last}$ fold product of  terms like $[1 + A_1(X_{j})]$, and bounding the expectation of the linear and cross products of $A_1(X_i)$'s we can show  convergence.

Let $n_H$ be the cardinality of $\Theta_H$.
By condition [B3], $E[A_1(X_i)+1]  \leq 1 + n_H c_u B_1 b/a = B_u <\infty$, where $E[\frac{p(X_i|\theta_l)}{p(X_i|\theta_k)}]<B_1$. 
Then, by condition [B2] any $m$-fold product $E[\prod_{i_1\neq\dots\neq i_m} A_1(X_{i_1})] < (B'_u)^m; 0<B'_u<\infty, B'_u=b_0B_u$. 

Expanding as product of ratios of successive marginals, 
\[\frac{m_{i-1}(X_i)}{m_{i_{last}}(X_i)}=\prod_{j=i_{last}+1}^{i-1}\left(1+w_j(\frac{\int p(X_j|\theta)p(X_i|\theta)f_{j-1}(\theta)d\theta}{m_{f_{j-1}}(X_i)m_{f_{j-1}}(X_j)}-1)\right).\]
When $X_j \notin \chi_c,$  $\frac{\int p(X_j|\theta)p(X_i|\theta)f_{j-1}(\theta)d\theta}{m_{f_{j-1}}(X_i)m_{f_{j-1}}(X_j)}\leq \frac{\int p(X_j|\theta_1)p(X_i|\theta)f_{j-1}(\theta)d\theta}{m_{f_{j-1}}(X_i)p(X_j|\theta_2)}\leq A_1(X_j)$, where $p(X_j|\theta)$ is minimized and maximized at $\theta_2$ and $\theta_1$, respectively. When $X_j\in \chi_c$, then  we have $\frac{\int p(X_j|\theta)p(X_i|\theta)f_{j-1}(\theta)d\theta}{m_{f_{j-1}}(X_i)m_{f_{j-1}}(X_j)}\leq \frac{b}{a}\leq A_1(X_j)$. 

Hence,
\[|w_i\frac{m(X_i)}{m_{i-1}(X_i)}(\frac{m_{i-1}(X_i)}{m_{i_{last}}(X_i)}-1)|\leq w_i (\sum_{j_1\in \mathscr{S}_i}w_{j_1}A_1(X_i)(A_1(X_{j_1})+1)+\]\[\sum_{j_1\neq j_2\in \mathscr{S}_i}w_{j_1}w_{j_2}A_1(X_i)(A_1(X_{j_1})+1)(A_1(X_{j_2})+1)+\]\[\dots
+\sum_{j_1\neq j_2\cdots \neq j_{m}\in \mathscr{S}_i}w_{j_1}w_{j_2}\cdots w_{j_m}A_1(X_i)(A_1(X_{j_1})+1)(A_1(X_{j_2})+1)..(A_1(X_{j_m})+1)+..\]
where $\mathscr{S}_i$ be the set $\{(i_{last}+1),..,i-1\}$. 

Coefficients for the $m$-fold products $A_1(X_{j_1})A_1(X_{j_2})\cdots A_1(X_{j_m});$ $j_1\neq j_2\neq\cdots\neq j_m$ are bounded by $w_i\sum_{t=m}^\infty w_{i_{last}}^t < c_1w_{i_{last}}^{(m+1)}$, for some $c_1>0$. The expectations of each of those $m$-fold product  bounded by ${B'_u}^m$. Let  $d_i=i-2-i_{last}$. Then there are ${d_i \choose m}$ many terms consisting of $m$-fold products of $A_i$. 

As $d_i = o({i}^{1/K+\epsilon})$, for any $\epsilon>0$.  Choosing   $N_0$ and  $K$ large enough, we have $w_{i_{last}}< C' i^{-\alpha_1}$, when  $i>N_0,C'>0$  and $.5 + 1/K <\alpha_1'< \alpha_1 <\alpha$, and we have
\[\sum_{i>N_0}E[|S^{\Delta}_i|]\leq C'\sum_i \sum_{m=2}^{d_i} i^{-m(\alpha'_1-1/K)}(B'_u)^m  < \infty.\]
By Proposition~\ref{pos_sum}, we have the result. 

\noindent \underline{\it Combining the parts:}
Having established that $\sum_{j=1}^{L_n} S^{(n)}_{v,j},\sum_{i=1}^n E_i,\sum_{i=1}^nS^{\Delta}_i, \sum \delta_i$ all converge with probability one to finite quantities, we essentially follow the arguments given in \citet{tokdar2009consistency}  . Since $K^*_i>0$, we have $K^*_i$ converging to zero in a sub sequence with probability one, as LHS in \eqref{wkkl1} is finite, because otherwise from the fact $\sum w_i=\infty$, RHS will be $-\infty$. 

Hence,  as $K_n\geq 0$, $\sum w_iK_i^*$ has to converge, as all other sequences converge and therefore $K_i^*$ converges to zero in a subsequence almost surely. Now, using  finiteness of $\Theta$, the proof follows as $K_n$ converges and over that subsequence $f_n$ has to converge to $f$, as otherwise $K_n^*$ cannot converge to zero in that subsequence (\citet{ghosh2006convergence}).

In particular, if  $K^*_{n_k(\omega)}\rightarrow 0$ and if $f_{n_{k_j}}\rightarrow \hat{f}\neq f$ in some sub-subsequence of $n_k$ then corresponding marginal $m_{n_{k_j}}$ converges to $\hat{m}\neq m$ weakly. But, as $K^*_{n_k}\rightarrow 0$, $m_{n_k}$ converges to $m$ in Hellinger distance and also weakly, which is a contradiction. Hence, $f_{n_{k_j}}\rightarrow f$ and therefore $K_{n_{k_j}}\rightarrow 0$. Therefore, $K_n\rightarrow 0$ and  we have $f_n\rightarrow f$ with probability one.
\end{proof}

In general a bigger value of $\rho$ in \eqref{strdep} would indicate stronger dependence among $X_i$ and hence the convergence rate will be expected to be slower for bigger $\rho$. Even though it is hard to write explicitly how  $\rho$ affects the convergence, some insight can be gleaned by studying the different components of the decomposition of the KL divergence and their convergence. 

 Under the setting of Theorem \ref{wdependent}, for convergence of $K_n$, we need for $n>m$  the Cauchy increments, $|K_m-K_n|$ to go to zero. 
From equation \eqref{wkkl1}, the Cauchy difference is essentially made up of a martingale difference sequence, an increment of non-negative  sequence, and other terms that constitute error terms. The main term in $|K_m-K_n|$  that is influenced directly by the dependence parameter  is $\sum_{j=1}^{L_n} S_{M,j}^{(m,n)}$. 
Following calculations of bounds on $M_i$ in \eqref{wdep_decmp}, we have  \[\sum_{j=1}^{L_n} S_{M,j}^{(m,n)} \leq \sum_{i=m}^n w_i\int(\frac{m(x)}{{m}_{{i}_{last}}(x)}-1)m(x)\nu(dx)+\sum_{l=m^{1/K}}^{(n+1)^{1/K}} w_l\tilde{b}_0l^K\rho^{l-1}  \]
for some $\tilde{b}_0>0$. 
Thus, convergence of the conditional mean sequence, and hence the KL sequence is expected to be slower for larger values of $\rho$. This intuition is  reinforced in  Example~1 where   the AR coefficient $r$ plays the role of the dependence parameter in \eqref{strdep}. Numerically, 
we see that   convergence is faster for smaller $r$ and slower  with larger values of $r$. 

As mentioned earlier, consistency of the recursive algorithm can be established in much more generality, even under mild dependence, provided we can assume a slightly stronger condition. The following condition is  stronger than [B2], but maybe more readily verifiable.  
\begin{itemize}
\item[] $\rm{B2}'$. $sup_{\theta_1,\theta_2,\theta_3 \in \hat{\Theta}} E_{\theta_3}[( \frac{p(X_i|\theta_{1})}{p(X_i|\theta_{2})})^j]<b_3^j$, for some $b_3>0$,
\end{itemize}
The condition is sufficient for establishing consistency in the non-finite support with dependent data. However, it need not be necessary. It may not hold in some cases where the support of $X$ is unbounded. For example, the condition does not hold in normal location mixture with dependent data but it does hold when the a truncated normal kernel is used. 

In the proof of Theorem~\ref{wdependent}, one could work with $[\rm{B2}']$ instead of [\rm{B2}]. 
\begin{rmrk}
If condition $[\rm{B2}]$ is replaced with condition $ [\rm{B2}']$ in the statement of Theorem~\ref{wdependent}, the conclusions of Theorem~\ref{wdependent} continue to hold. 
\end{rmrk}
The proof of the corollary is straight-forward and is omitted. The condition $[\rm{B2}']$ is in essence equivalent to [B2] and [B3], when the kernel is bounded and bounded away from zero. 

\begin{rmrk}
If  $inf_\theta p(x|\theta)>0$ and $sup_\theta p(x|\theta)<\infty$ then \mbox{\rm{B2, B3}} and $\mbox{\rm{B2}}'$ are satisfied. 
\label{pos_den}
\end{rmrk}

%One of the limitations of Theorem~\ref{wdependent} is the finiteness of the parameter space. However, if condition B2 is replaced with B2', consistency of the solution to the predictive resursion holds for more general compact parameter spaces.  The following theorem gives the general result. 

\begin{thm}
Let $\{X_i\}$ be sequence of random variables satisfying \eqref{strdep} where $X_i$ has a fixed marginal density $m(\cdot)$ given by $m(x) = \int_{\Theta} p(x | \theta) f(\theta) d\theta$ and the support of $f$, $\Theta$, is a compact subset of the corresponding  Euclidean space.
Assume the initial estimate $f_0$ in \eqref{newton} has the same support as $f$.  Let $F_n$ and $F$ denote the cdfs associated with $f_n$ and $f$, respectively. Then under ${\rm{B1},\rm{B2}',\rm{B3,B4}}$, \mbox{\rm{C1--C3}}, the estimate $F_n$ from equation \eqref{newton} converges to $F$ in weak topology with probability one, as $n$ the number of observations goes to infinity. 
\lb{wdependent2}
\end{thm}
\begin{proof}
Given in the Appendix.
\end{proof}

\subsection{{ Misspecification of support }}
The predictive recursion algorithm requires specification of the support of the mixing density. The support of $f$ could be misspecified, in which case the solution $f_n$ cannot converge to the true density $f$, however one could still investigate convergence for the sequence $f_n$. Let the support of the initial density in the predictive recursion \eqref{newton} be $\Theta_0$, a compact set  possibly different from $\Theta$, the support of the true mixing density $f$. 
Let $\mathcal{F}_0$ be the class of all densities with the same support as 
$f_0(\theta)$. Specifically, let
\beqn
\mathcal{F}_0 &=& \{{\hat{f}}: {\hat{f}} \mbox{ is a density supported on } \Theta_0\} \nonumber \\
\mathcal{M}_0 &=& \{\hat{m}: \hat{m}(x)=\int_{\Theta_0} p(x|\theta)\hat{f}(\theta)\mu(d\theta); {\hat{f}} \in \mathcal{F}_0\}.
\lb{F0}
\eeqn
Let
\be
\tilde{k}  = \underset{{\hat{m}}\in \mathcal{M}_0}\inf KL(m,\hat{m}).
\lb{ftilde}
\ee
Thus, assuming uniqueness, the  minimizer ${\tilde{m}}\in \mathcal{M}_0$ is the {\it information projection} of $m$ to $\mathcal{M}_0$.  This minimizer is related to the  Gateaux differential   of KL divergence (\citet{patilea2001convex}). For independent observations misspecification of the support has been addressed in the literature; for example see \citet{ghosh2006convergence}, \citet{martin2008stochastic}. Using the decomposition for dependent case in the proof of Theorems~\ref{wdependent2}, \ref{wdependent22}, we extend the result for dependent cases. 
Let ${\tilde{K}}_n^* = K_n^* - \tilde{k}$ where  $K_n^*$  is defined in \eqref{kl2}.
\begin{thm}
Let $\{X_i\}$ be sequence of random variables satisfying \eqref{strdep} and $X_i$ have a fixed marginal density $m(\cdot)$ given by $m(x) = \int_{\Theta} p(x | \theta) f(\theta) d\theta.$ 
Assume  $\Theta$ is a compact and  $f_0$ in \eqref{newton} belongs to $\mathcal{F}_0$ defined in \eqref{F0}.  Assume, there exists a unique $\tilde{f}$ such that $K(m, \tilde{m}) = \tilde{k}$ where $m_{\tilde{f}}=\tilde{m}(x) = \int_{\Theta_0} p(x |\theta)\tilde{f}(\theta)  \mu(d\theta)$.
Assume ${\rm{B1},\rm{B2}',\rm{B3,B4}}$, \mbox{\rm{C1--C3}}. Then,
\[  \lim_{n\to\infty} {\tilde{K}}_n^*  =  0, \mbox{  w.p.1.}\] 
\lb{wdependent22}
\end{thm}
Note that since $\mathcal{F}_0$ is convex, so is the set of corresponding marginals. Thus, uniqueness of $\tilde{f}$ is guaranteed if the model is identifiable.

The proofs of Theorems~\ref{wdependent22} is given in the Appendix following the proofs of analogous result Theorems~\ref{mdependent22} for $M$-dependent as the proofs involve similar arguments.
\subsection{Convergence rate of the recursive algorithm}
Convergence of PR type algorithms has been explored in recent work such as  \citet{martin2009asymptotic}  where the fitted PR marginal shown to be in $\sim n ^{-1/6}$ radius Hellinger ball around the true marginal almost surely. \citet{martin2012convergence} establishes a better bound for finite mixture under misspecification. PR convergence rates are  similar to nonparametric rate in nature and in current literature does not have the minimax rate (up to  logarithmic factors) such as the rates  shown in \citet{ghosal2001entropies} and subsequent developments. The  convergence rate calculations (\cite{martin2009asymptotic}, \cite{martin2012convergence}) follow from the    super martingale convergence theorem from Robbins-Siegmund (\citet{robbins1971convergence}, \citet{lai2003stochastic}) for independent cases and yield a rate similar to \citet{genovese2000rates}. 

In the presence of dependence, we show the rate calculation assuming a  faster rate of decay for the weights $w_i$. Instead of condition B1 we will assume
\begin{itemize}
\item[$\rm{B1}'$] $w_i \downarrow 0$, and $w_i \sim i^{-\alpha}$,  $\alpha \in (0.75,1]$.
\end{itemize}
We use the subsequence construction technique described in  \ref{martcon} to do the rate calculations. Following  the arguments given in the proofs of  Theorems \ref{wdependent} and \ref{wdependent2}, we establish almost sure concentration in $n^{-(1-\gamma)/2}$ radius Hellinger ball around the true marginal for $\gamma>\alpha$. The rate is slower than the rate for the independent case. Let  $H^2(f,g)=\int (\sqrt f-\sqrt{g})^2 d(\cdot)$ the squared Hellinger distance between densities.

\begin{thm}
Assume the conditions of Theorem \ref{wdependent22} hold with \rm{B1} replaced with $\rm{B1}'$. Then 
\[ n^{-\gamma + 1}\tilde{K}_n^* \rightarrow 0, \mbox{  with probability one} \]
for some $\gamma \in (\alpha,1)$. Moreover, if $\frac{m}{m_{\tilde{f}}}$ is bounded away from zero and infinity, then  $H(m_{\tilde{f}},m_n)=o(n^{-(1-\gamma)/2})$.
\lb{rate_thm2}
\end{thm}
\begin{proof}

From the proof of Theorem \ref{wdependent2} and \ref{wdependent22} (in Appendix), $\sum w_i\tilde{K}^*_{i_{last}}$ converges where $i_{last}=\psi(l-1,j)$ if $i=\psi(l,j)$. By construction $i-i_{last}\preceq i^{1/K}$ and $\frac{w_{i_{last}}}{w_i}=O(1)$.  Hence, $\sum w_i\tilde{K}^*_{i-1}$ converges almost surely. Let, $a_i=\sum_{j=1}^iw_j, a_0=0$. Then,  from \ref{wkkl2}, 
\begin{eqnarray}
a_n\tilde{K}_n^*-a_0\tilde{K}_0^*=\sum_{i=1}^nw_i\tilde{K}^*_{i-1}+\sum_{j=1}^{{L_n}} \hat{S}_{v,j}^{*,(n)}-\sum_{j=1}^{L_n} \hat{S}_{M,j}^{*,(n)}+\sum_{i=1}^na_iS^{\Delta*}_i+\sum_{i=1}^na_i E^*_i . \nonumber\\
\lb{wkkl3}
\end{eqnarray}
Here
\begin{eqnarray*}
 \hat{S}_{v,j}^{*,(n)}&=&\sum_{i'=1}^{N_{j,n}}\hat{v}^*_{\Psi(i',j)},\\
\hat{S}_{M,j}^{*,(n)}&=&\sum_{i'=1}^{N_{j,n}}a_{\Psi(i',j)}w_{\Psi(i',j)}M^*_{\Psi(i',j)},\\
\hat{v}_{i}^*&=&a_iw_i(1-\int\frac {h_{i,X_i}(x)}{m_{i-1}(x)}m(x)\nu(dx))-a_iw_iE [1-\int_{\chi}  \frac{h_{i,X_i}(x)}{m_{i-1}(x)}m(x)\nu(dx) |\mathcal{F}_{last,i}].
\end{eqnarray*}

Writing $w_i'=a_iw_i\sim i^{-\alpha_2'}$, $\alpha_2'=2\alpha-1>0.5$, from the proof of Theorem \ref{wdependent2}, it can be shown that $\sum_{i=1}^na_iS^{\Delta*}_i$ , $\sum_{i=1}^na_i E^*_i $, $\sum_{j=1}^{{M_n}} S_{v,j}^{*,(n)}$ converge almost surely. We already established that $\sum w_i\tilde{K}^*_{i-1}$ converges almost surely. Hence,  $\sum_{j=1}^{M_n} \hat{S}_{M,j}^{*,(n)}$ has to converge almost surely and would imply that $a_{\Psi(i',j)}w_{\Psi(i',j)}M^*_{\Psi(i',j)}$ goes to zero over a subsequence.

As,  $a_n\tilde{K}^*_n$ converges to a finite number  for each $\omega$ outside a set of probability zero, and $a_n \sim n^{1-\alpha}$, then $n^{1-\gamma}\tilde{K}_n^*$ goes to zero with probability one for $\gamma>\alpha$. As $\frac{m_{\tilde{f}}}{m}$ is bounded away from zero and infinity, $KL(m_{\tilde{f}},m_i) \sim \int_{\chi}\log\frac{m_{\tilde{f}}(x)}{m_i(x)}m(x)\nu(dx)$, the conclusion about the Hellinger distance follows as  Kullback Leibler divergence is  greater than squared Hellinger distance.

\end{proof}

\subsection{Finite Mixture of Markov processes and other examples}
We consider a few examples of dependent processes, e.g. mixtures of Markov processes,  where the marginal density is a mixture. Uniqueness of such representation has been studied extensively in the literature. Some  of the earliest work can be found in \citet{freedman1962mixtures}, \citet{freedman1962invariants}.  Also see related work in \citet{diaconis1980finetti}.  We here look at variants of stationary autoregressive  processes of order one, AR(1), plus noise and Gaussian processes with constant drift. While simulation is done for  different sample sizes, our main objective is  to study the effect of dependence on convergence of the PR algorithm and not the rate of convergence. In fact, we use $w_i=1/(i+1)$ throughout for which $\alpha = 1$. Thus, condition $\rm{B1}'$ and hence Theorem  ~\ref{rate_thm2} are not applicable for these examples.\\ 

The condition given in \eqref{strdep} is a sufficient condition, and if the  marginal, $m(X_i)$ is a finite mixture of marginals of latent Markov processes which individually satisfy \ref{strdep},  and B2 then   a similar but slightly weaker condition will hold for $m(\cdot)$, and Theorem \ref{wdependent} will hold for the mixture. In particular, if we have  $X_i=I_{i,1}Z_{i,1}+\dots+I_{i,k}Z_{i,k}$ where  $Z_{i,j}$'s are independent Markov processes with stationary marginal densities $m^{(1)}(\cdot),\dots,m^{(k)}(\cdot)$ and $I_{i,1},\dots,I_{i,k}\sim Multinomial(1,p_1,\dots,p_k)$ and the marginal distribution $m^{(l)}(x)$'s are  parametrized by $\theta_l$'s( i.e. of the form $p(x|\theta_l)$), we can state the following result. 

\begin{proposition}
Let 
$X_i=I_{i,1}Z_{i,1}+\dots+I_{i,k}Z_{i,k}$ where  $Z_{i,j}$'s are independent Markov processes with stationary marginal densities $m^{(1)}(\cdot),\dots,m^{(k)}(\cdot)$ and $I_{i,1},\dots,I_{i,k}\sim Multinomial(1,p_1,\dots,p_k)$. Suppose, for each $l=1,\ldots,k,$ there exists  $c_0>0$ and $0 < \rho <1$ such that
\begin{eqnarray}
sup_i\int(\frac{m^{(l)}(X_{i+n}|X_{1:i})}{m^{(l)}(X_{i+n})}-1)^2 m(X_{i+n})m(X_{1:i})\nonumber \\
=sup_i\int(\frac{m^{(l)}(X_{i+n}|X_i)}{m^{(l)}(X_{i+n})}-1)^2 m(X_{i+n})m(X_i)\leq& c_0\rho^{2n}.
\label{str_dep_mrkov}
\end{eqnarray}  
Then Theorem \ref{wdependent} holds for the mixture process.
\label{markov_mix}
\end{proposition}

\begin{proof}
By assumption, the marginal of $X_i$ is  $m(X_i)=\pi_1m^{(1)}(X_i)+\cdots+\pi_km^{(k)}(X_i)$. Let $\mathscr{E}^{i_1}_{i_2,i',j}$ be the event, $I_{i_1,l}=1$, for $l=j$ and the last time $I_{i,j}=1$, for $i\leq i'$ for $i=i_2$. Then
\[(\frac{m(X_{i+n}|X_{1:i})}{m(X_{i+n})}-1)^2 = 
(\frac{\sum_l\pi_lm^{(l)}(X_{i+n})+\sum_{i'=1}^i \sum_l\Delta^{(l)}_{i+n | i:i'} p(\mathscr{E}^{i+n}_{i',i,l}|X_{1:i})}{\sum_l\pi_lm^{(l)}(X_{i+n})}-1)^2.\]
where $\Delta^{(l)}_{i+n | i:i'} = m^{(l)}(X_{i+n}|X_{1:i'})-m^{(l)}(X_{i+n}).$
Hence, 
\begin{eqnarray}
(\frac{m(X_{i+n}|X_{1:i})}{m(X_{i+n})}-1)^2 &\leq& [\sum_l\pi_l^{-2}\sum_{i'=1}^i |\frac{m^{(l)}(X_{i+n}|X_{1:i'})-m^{(l)}(X_{i+n})}{m^{(l)}(X_{i+n})}|p(\mathscr{E}^{i+n}_{i',i,l}|X_{1:i})]^2\nn\\
&\preceq& max_l \sum_{i'=1}^i (\frac{m^{(l)}(X_{i+n}|X_{1:i'})}{m^{(l)}(X_{i+n})}-1)^2.
\label{str_dep2}
\end{eqnarray}
From  the calculation following equation \eqref{wdep_decmp}, it follows that $E[\delta_i^{(i)}]\preceq i_{last} \rho^{i-i_{last}}$ with $\rho$  given in equation \eqref{str_dep_mrkov} and  $i_{last}$ is as defined in the proof of Theorem \ref{wdependent}, where $\preceq$ denotes less than equal to up to a constant multiple. For  $i$ between $l^K$ and $(l+1)^K$, $i-i_{last}=l$. Hence, we have
$\sum_i E[\delta_i^{(i)}]\preceq l^{2K}\rho^{l-1}<\infty$. 

Similarly, B2 can be verified by conditioning on the indicators and rest of the argument from Theorem \ref{wdependent} follows.
\end{proof}

\noindent{\it Example 1. AR(1).} As an example of mixture of independent Markov processes, we consider a mixture where one component is a Gaussian white noise and other component is a stationary AR(1) process. Let $X_{1,i} \sim AR(1)$  with marginal $N(0,1)$.  and let $X_{2,i} \stackrel{iid}{\sim} N(2.5,1)$,  independent of $X_{1,i}$.  Let $X_i=I_iX_{1,i}+(1-I_i)X_{2,i}$, where $I_i$s are independent Bernoulli(0.3). We consider different values for the AR(1) parameter, $r$,   and investigate the effect on the convergence. To have the stationary variance of the AR(1) process to be one we choose the innovation variance to be equal to $\sqrt{1 -r^2}$. For the AR(1) process, the dependence parameter in condition \eqref{strdep}, $\rho,$ is a monotone function of $r$. Specifically, we use $r \in \{0.3, 0.7, 0.99, 0.999\}$.  A typical example is given in Figure \ref{AR1}, where we see the higher the value of $r$ the slower the convergence. While for moderate values of $r$, the effect is negligible, for strong dependence,  the effect on convergence is very pronounced. In this example for $f_0$ starting mixing probability $p=0.7$ has been used where the true mixing probability is 0.3.
\begin{figure}
\centering
\includegraphics[width=2in,height=1.45in]{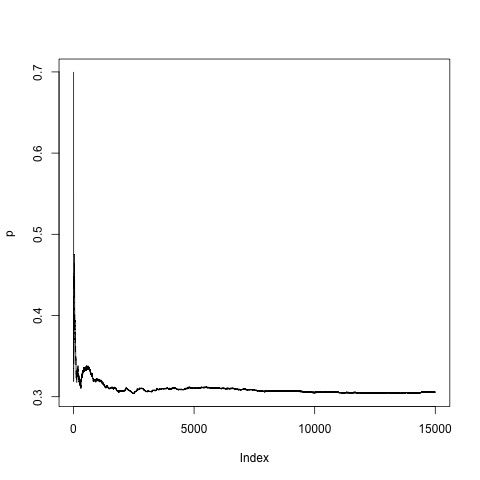}
\includegraphics[width=2in,height=1.45in]{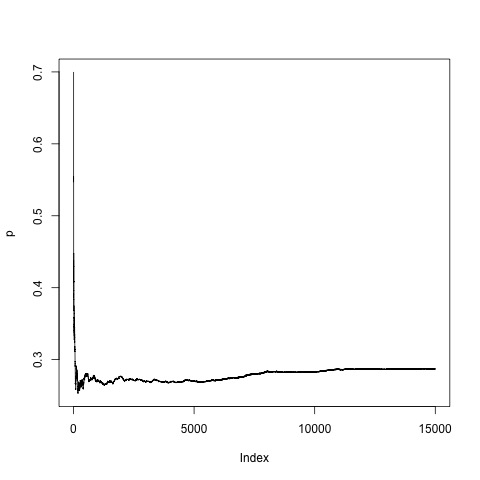}\\
\includegraphics[width=2in,height=1.45in]{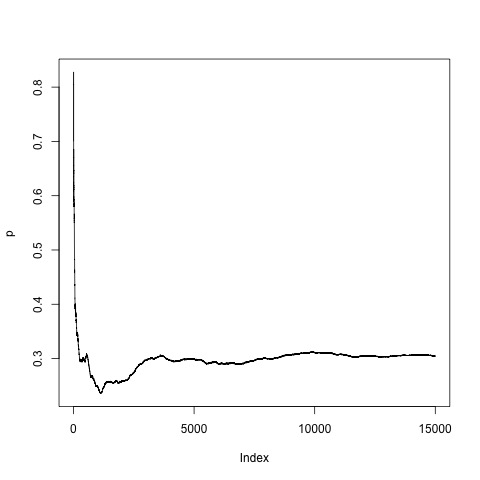}
\includegraphics[width=2in,height=1.45in]{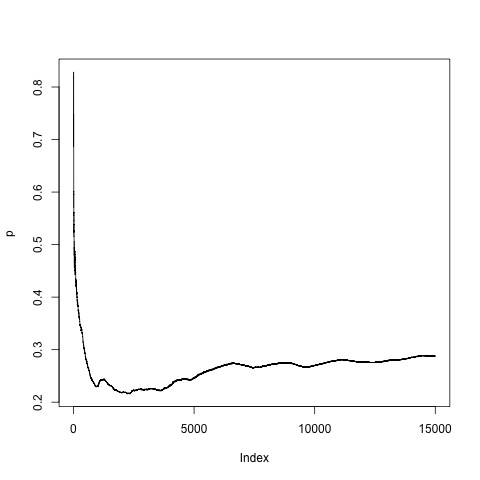}\\
\caption{ The AR 1 example for the general dependent case. Top panel shows $r=.3$ and $r
=.7$ in left  and right panel, respectively. Bottom panel shows $r=.99$ and $r=.999$ in left  and right panel, respectively. True value of $p$ is .3 . The  heavily dependent cases $r=.99$ and $r=.999$ show slower convergence. }
\lb{AR1}
\end{figure}

%\subsubsection{Calculation for the AR(1) mixture}
%Here, we show the result for the case when a AR(1) process is mixed with an independent process as in Example 1, Section 3.
%
% \underline{\it Showing equation \eqref{str_dep_mrkov}.} 
%\it We have $m^{(1)}(\cdot)$ the density of  $AR(1)$ process, and $m^{(2)}(\cdot)$ correspond to $N(\mu,1)$ independent over $i$. The conditional densities evaluated at the observed, ${m^{(1)}(X_{i+n}|X_1,\dots X_{i'})}\sim N(r^{n'}X_{i'},1-r^{2n'})$ where $n'=n+i-i'$.
%Hence, for $i'\leq i$,
%\begin{eqnarray*}
%\int[\frac{m^{(1)}(X_{i+n}|X_{i'})}{m^{(1)}(X_{i+n})}-1)]^2m^{(1)}(X_{i+n})=1-d_ne^{c_nX_i'+\tilde{c}_n{X_i'}^2},\\
% \int[\frac{m^{(1)}(X_{i+n}|X_{i'})}{m^{(1)}(X_{i+n})}-1)]^2m^{(2)}(X_{i+n})=1-d'_ne^{c'_nX_i'+\tilde{c}'_n{X_i'}^2}
% \end{eqnarray*}
%where $d_n,d_n'=1+O(r^{n'}); max\{ |c_n|,|c_n'|,|\tilde{c}_n|, |\tilde{c}'_n|\}\preceq r^{n'}; n'=n+i-i'$. Hence, computing another expectation with respect to $m^{(1)}(X_{i'})$, $m^{(2)}(X_{i'})$ the result follows.

Here it is easy to see that the individual components satisfy \eqref{strdep}.  Thus, the PR algorithm based on $\{X_i\}$ will consistent for the marginal $m(X_i)$ provided $\{X_i\}$ satisfy the  moment conditions assumed in Theorem~\ref{wdependent}. This part follows from conditioning on the indicators and then calculating $E[(\prod_{l=1}^{n_1}\frac{p(X_{j_l}|\theta_{l_1})}{p(X_{j_l}|\theta_{l_2})})^2]$ given the indicators by recursively conditioning on the observed $X_i$'s from the earlier indices. In particular, it will be sufficient  to show for the case when $n_1=n$, $j_i=i$, and the values of $I_i=1$ for all $i$ that is $X_1,\dots,X_n$ comes from the latent $AR(1)$ process.
 Note that $(\frac{p(X_{j_l}|\theta_{l_1})}{p(X_{j_l}|\theta_{l_2})})^2=c'_{\theta_{l_1},\theta_{l_2}}e^{c_{\theta_{l_1},\theta_{l_2}}X_i}$ where $|c_{\theta_{l_1},\theta_{l_2}}|,|c'_{\theta_{l_1},\theta_{l_2}}|$ are uniformly bounded (say by $\tilde{c}>0$). 
 For convenience we write $c_{\theta_{l_1},\theta_{l_2}}=c_{l,1,2}, c'_{\theta_{l_1},\theta_{l_2}}=c'_{l,1,2}$. 
 
  Let $j_l=n, \theta_{l_1}=\theta_{n_1}, \theta_{l_2}=\theta_{n_2}$. Let $\mathscr{E}_{1,\dots,n}$ denote the event that $I_{i}=1, i=1,\dots,n $.  Hence,
  \[E[ [\frac{p(X_{n}|\theta_{n_1})}{p(X_{n}|\theta_{n_2})}]^2|X_1,\dots,X_{n-1}, \mathscr{E}_{1,\dots,n}]=e^{r c_{n,1,2}X_{n-1}+\frac{c^2_{n,1,2}(1-r^2)}{2}}c'_{n,1,2}.\]
  Similarly, 
  \begin{eqnarray*}
E[ [\frac{p(X_{n}|\theta_{n_1})}{p(X_{n}|\theta_{n_2})}\frac{p(X_{n-1}|\theta_{{(n-1)}_1})}{p(X_{n-1}|\theta_{{(n-1)}_2})}]^2|X_1,\dots,X_{n-2},\mathscr{E}_{1,\dots,n}]\\ \leq\tilde{c}^2e^{(c_{n-1,1,2}+rc_{n,1,2})X_{n-2}+\frac{(r c_{n,1,2}+c_{n-1,1,2})^2(1-r^2)+c^2_{n,1,2}(1-r^2)}{2}}.
 \end{eqnarray*}
Applying the bounds recursively we have the result.

\noindent{\it Example 2. A continuous mixture of AR(1).} Next we consider a continuous mixture in the mean with error term following an AR(1) dependence. Let  $X_i=\theta_i+Z_i$ where $Z_i$ follows an AR(1) model with  parameter  $r = 0.7$. Let $\theta_i \stackrel{iid}{\sim} TN(0,1,-3,3)$, a standard normal distribution truncated to [-3,3].  Using a uniform[-3,3] as  $f_0$ in \eqref{newton}, the fitted values are given in the left box in Figure \ref{armix0}. The   true mixing density is given in solid blue. Solid  and dashed black show the fit using $n=1000, 2000$, respectively.  

The following argument verifies the conditions for consistency for the PR algorithm  for a mean mixture of  AR(1) process, $X_i=\theta_i + Z_i$. Here  $Z_i$ follows $AR(1)$ and $\theta_i$  are i.i.d with density $f(\theta)$,  where $f(\theta)$ is bounded and  bounded away from zero and has compact support.
We write,
\[\frac{m(X_{i+n}|X_{1:i})}{m(X_{i+n})}=\frac{\int p_Z(X_{i+n}-\theta|\theta_i=\theta',X_{1:i})f(\theta)f_{i|X_1\dots,X_i}(\theta')d\theta d\theta'}{\int p_Z(X_{i+n}-\theta)f(\theta)d\theta}.\]
Here, $p_Z(X_{i+n}-\theta|\theta_i=\theta',X_{1:i})=p_Z(X_{i+n}-\theta|X_i-\theta')$, and $f_{i|X_1\dots,X_i}(\cdot)$ is the conditional density of $\theta_i$ given $X_1,\dots,X_i$.  Hence,
\begin{eqnarray*}
|\frac{m(X_{i+n}|X_{1:i})}{m(X_{i+n})}-1|&=&|\frac{\int \{p_Z(X_{i+n}-\theta|X_i-\theta')-p_z(X_{i+n}-\theta)\}f(\theta)f_{i|X_1\dots,X_i}(\theta')}{\int p_Z(X_{i+n}-\theta)f(\theta)}|\\
&\leq& F_1(X_{i+n}){\int |\frac{p_Z(X_{i+n}-\theta|X_i-\theta')}{p_z(X_{i+n}-\theta)}-1|f(\theta) f_{i|X_1\dots,X_i}(\theta')},
 \end{eqnarray*}
 where $F_1(X_{i_n})\preceq e^{\tilde{c} X_{i+n}}+ e^{-\tilde{d} X_{i+n}}, \tilde{c},\tilde{d}>0$, using the fact that the support of $f(\cdot)$ is bounded and $f(\cdot)$ is bounded away from zero and infinity.  Hence, all the moments of $F_1(X_{i+n})$ are finite  (in particular,   $E[(F_1(X_{i+n}))^{a_0'}]<\infty$ for any $a_0'>0$). The result  follows by noticing that $ \int p_Z(X_{i+n}-\theta)f(\theta)>e^{-\frac{X^2_{i+n}}{2}}e^{-a_0|X_{i+n}|}b_0$; where $ a_0,b_0>0$ are constants. It should be noted that  $X_{i+n}=Z_{i+n}+\theta_{i+n}$, where $\theta_{i+n}\sim f(\theta)$ and bounded. 
Using a similar argument, writing the joint density of $X_1,\dots,X_n$ in convolution form, $f_{i|X_1\dots,X_i}(\theta')\leq F_2(X_i,X_{i-1})$ where all the moments of $F_2(\cdot,\cdot)$ are finite. 

Now, $\frac{p_Z(X_{i+n}-\theta|X_i-\theta')}{p_z(X_{i+n}-\theta)}=c_n''e^{c_n'X_{i+n}^2+d_n'X_{i+n}X_i+e_n'}$, where $c_n',d_n',e_n'=O(r^n); c_n''=1+O(r^n)$. Hence, using Cauchy-Schwartz inequality we have,  
\[\int|\frac{m(X_{i+n}|X_{1:i})}{m(X_{i+n})}-1|^2m(X_{i+n})m(X_{1:i})\preceq r^n.\]
Hence the dependence condition is satisfied. 
Next we verify condition $[B2]$ for the mixture process.
For $i_1<\cdots<i_{n'}$, such that $\{i_1,\cdots,i_{n'}\}\in\{1,\cdots,n\}$, and  for $X_i=Z_i+\theta_i$, where $\theta_i\in [a,b]$, let $\theta_{i_l};l=1,\cdots,n'$, and  $\theta_{j_l};l=1,\cdots,n'$ such that $\theta_{i,l},\theta_{j,l}\in\hat\Theta\subset [a,b]$.  Choosing, $b$ large enough, and $a$ small enough,
\begin{eqnarray*}
\prod_{l=1}^{n'} \frac{p(X_{i_l}|\theta_{i_l})}{p(X_{i_l}|\theta_{j_l})} \leq e^{n'[(b-a)^2+b^2]}\prod_{l=1}^{n'} e^{-Z_{i_l}(\theta_{j_l}-\theta_{i_l})}
\end{eqnarray*}
where $Z_i$ follows $AR(1)$ with parameter $r$. Then by computing expectations recursively and 
noting that $Z_{i_{n'}}|Z_{i_1},\cdots,Z_{i_{n'-1}}\sim N(r^{i_l-i_{l-1}}Z_{i_{l-1}},1-r^{2(i_l-i_{l-1})})$, the moment condition is established using arguments similar to those in Example~1.

\noindent{\it Example 3. An irregular mixture with AR(1) error.} Consider the last example but for a  mixing distribution that has more structure. Specifically, let  $\theta_i\stackrel{iid}{\sim} 0.5\delta_{\{0\}} + 0.5 TN(4,1,-8,8)$. Here $\Theta=[-8,8]$. Using  a continuous uniform $f_0$ on [-8,8] the fitted densities are  given in the middle box of Figure \ref{armix0} for sample sizes  $n=500, 1000$, respectively. The algorithm converges to a mixture structure for $f$ with  probability around zero in the interval (-.5,.5) approximately equal to  0.45. The true mixing distribution is given by the solid line and the fit are given by the dotted line.  While the estimated density is showing bimodality,  a much larger sample size was also used to see explicitly the effect of large $n$. The right box in Figure \ref{armix0} shows the estimated mixture density for $n = 5000$. The estimate is markedly better around the continuous mode and the other mode is more concentrated around zero indicating recovery of the discrete part. How the convergence is markedly slower than that in Example~2. Verification for conditions of  consistency follows from the previous example. 
\begin{figure}
\centering
\includegraphics[width=1.6in,height=1.5in]{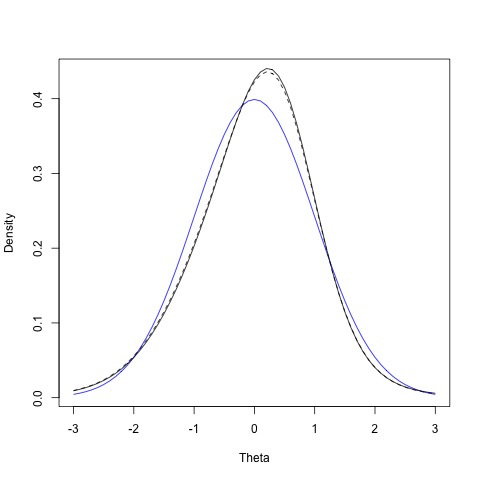}
\includegraphics[width=1.6in,height=1.5in]{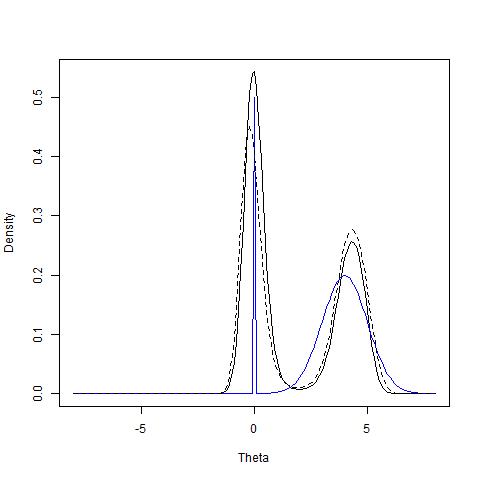}
\includegraphics[width=1.6in,height=1.5in]{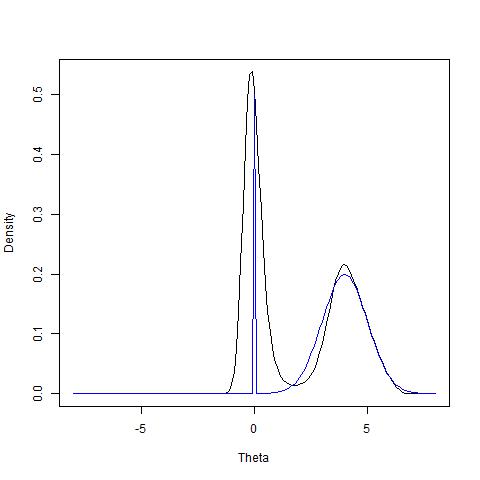}
\caption{ Examples with continuous  mixture of AR(1). Left panel corresponds to Example~2 with $\theta_i \sim TN(0,1,-3,3)$; middle and right panel are density estimates for Example~3.For Example~2,  $n=1000,2000$ (dashed and solid, respectively) .
For Example~3: $n=500,1000$ (dashed and solid respectively, in the middle box) and $n=5000$ in the right panel. True mixture is given by blue and the fitted ones are given by black.  }
\lb{armix0}
\end{figure}

\noindent{\it Example 4. Gaussian process paths.}

We consider two sub-examples with observations from a  latent Gaussian process with known covariance kernel. In first case, the observations are shifted by a fixed mean with some unknown probability. In the second case,   we consider a unknown constant drift along with the mean function. We observe from the marginal distribution.

\noindent {\it a) Zero drift.} Consider  a continuous time process observed at times $t_i, i=1,\ldots,n$,  where the observed process is the sum of  two independent latent processes.  Suppose, $X(t) \sim \mathscr{GP}(-1,\mathscr{K})$ with covariance kernel $\mathscr{K}(t_1,t_2)=.1e^{-(t_1-t_2)^2/10}$ and the observed process $Y(t_i) = Z_i a_1(t_i) + X(t_i)$ where  $a_1(t)=3$ and $Z_i \stackrel{iid}{\sim} Bernoulli(0.3)$ independent of $X(t)$. 
Using known support i.e support at $-1$ and $2$ for the mean we can plot the predictive recursion update for the probability at -1  in Figure \ref{drift0} left panel (true value is 0.7). Suppose, we start with a continuous support on $[-3,3]$ and uniform $f_0$ then the predictive recursion solution for $n=1000$ is given in right panel.\\
\begin{figure}
\centering
\includegraphics[width=2in,height=1.6in]{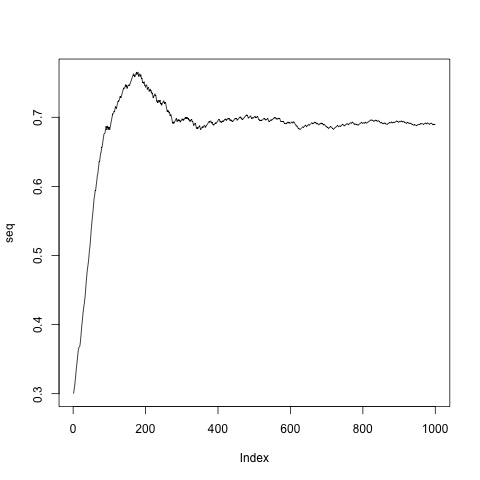}
\includegraphics[width=2in,height=1.6in]{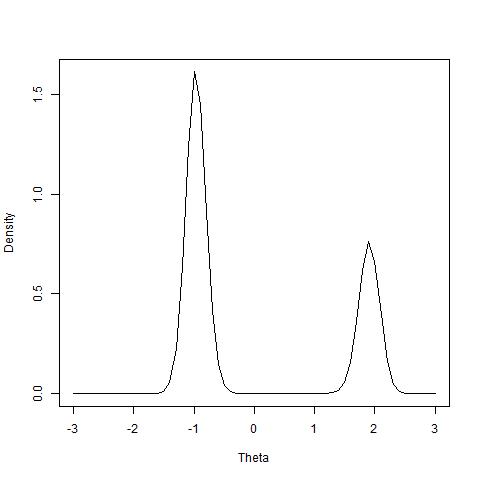}
\caption{Example 4 in Weak-dependent case. Left  panel shows the convergence of the mixing density under true discrete support where index in the X axis is the number of observations used. Right panel shows the fitted mixing density if $f_0$ is $Unif[-3,3]$ which gives mass around the true discrete support.}
\lb{drift0}
\end{figure}

\noindent {\it b) Constant drift.} A similar setting as in part (a) is considered with  $Y(t_i)=Z_ia_1(t_i)+X(t_i)$, $X(t) \sim \mathscr{GP}(0,\mathscr{K})$ and  $\mathscr{K}(t_1,t_2)=.1e^{-(t_1-t_2)^2/10}$,  where there is a drift given by $a_1(t)=\alpha+\beta t/100$ and the observed points are at $t_i$'s, $i=1,\cdots,n$, and $\alpha=5,\beta=2$. Figure \ref{drift02} top panel shows the fitted marginal and joint mixing densities for the slope and intercept parameters $\alpha$ and $\beta$  by the SA algorithm,  where initial density is uniform on the rectangle  $[-6,6]\times[-6,6]$. We have concentration around $0$ and $5$ for the intercept parameter, and, $0$ and $2$ for the slope parameter.  

 Figure  \ref{drift02} bottom panel shows the marginal fit when the starting density is uniform on  $[-3,3]\times[-6,6]$ and  the support does not contain the true intercept  parameter. In that case the mixing density for intercept concentrates around $0$ and $3$, corresponding to nearest KL distance point.
\begin{figure}
\vspace{-0.3in}
\centering
\includegraphics[width=1.3in,height=2in]{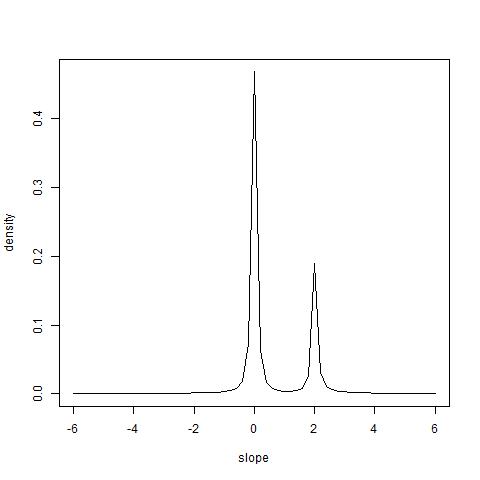}
\includegraphics[width=1.3in,height=2in]{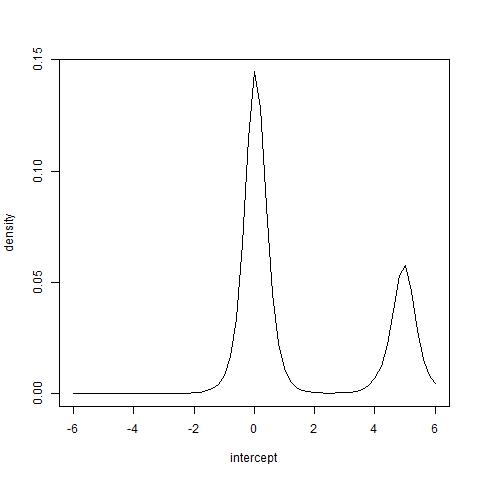}
\includegraphics[width=2.1in,height=2in]{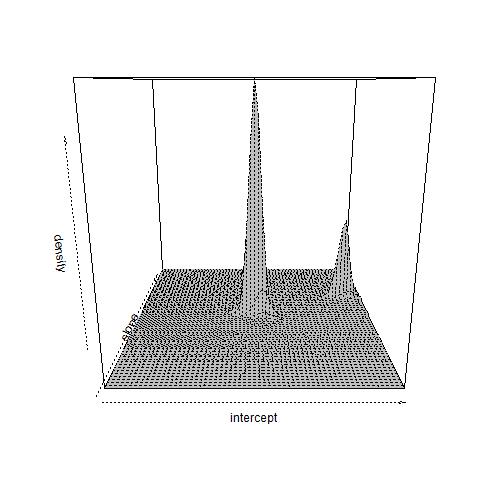}\\
\includegraphics[width=1.8in,height=1.4in]{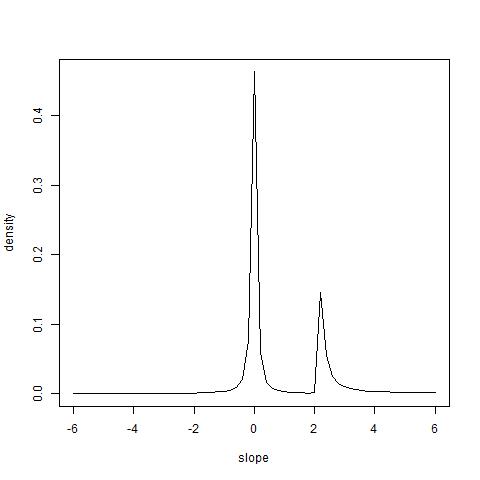}
\includegraphics[width=1.8in,height=1.4in]{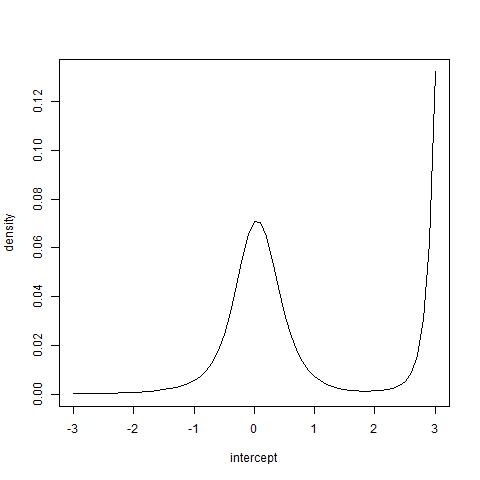}
\caption{Example 4 in Weak-dependent case. Gaussian process with constant drift. Top left, middle and right panel shows the estimated mixing densities for the slope, intercept and the joint mixing density. Bottom panel shows the marginal mixing densities for the case when the support for the initial distribution does not contain the true intercept. }
\lb{drift02}
\end{figure}

\section{$M$-dependent processes} 
An important subclass  of the weakly dependent processes are the $M$-dependent processes where $X_i$ and $X_j$ are independent if $|i-j|> M$ for some positive integer $M$.
Heuristically, if the process is $M$-dependent  for some  positive integer $M=q-1$, one expects the original PR algorithm to provide consistent estimates over $q$ different subsequences  where consecutive indices are $q$ apart and hence provides overall consistent estimator. Thus, one expects that the sufficient conditions assumed for general weakly dependent processes can be significantly relaxed. It is indeed the case and the proof for the $M$-dependent processes is also significantly different. Hence we present the case of $M$-dependent processes separately.

Consider an $M$-dependent  sequence $\{X_i\}_{i \in \mathbb{Z}}$ with fixed marginal distribution  $ X_i \sim m(x)$ of the form \eqref{marginal}.
From the definition of the process $X_i$, $X_{i+h}$ are independent for all $i$ and $h \geq q$ for some positive integer $q$. An example of such a  process would be  a $(q-1)$th order moving average, $MA(q-1)$, defined as 
\[  x_i = e_i + \psi_1 e_{i-1} + \cdots \psi_{q-1} e_{i - q +1}, \;\;\; i \in \mathbb{Z} \]
where $(\psi_1, \ldots, \psi_{q-1})$ are fixed parameters and $\{e_i\}$ are independent mean zero random variables with density $m_e(\cdot)$. Let the marginal density of $X_i$ be of the mixture form \eqref{marginal}. Note that the  $MA(q-1)$ process considered is stationary, but for the PR example we merely need the marginal density to not change with the index. 
As mentioned, the assumptions could be relaxed in the $M$-dependent case. The new assumptions are
\begin{itemize}

\item[A1] $w_n\downarrow 0$ as $n\uparrow \infty$; $\sum_{i=1}^\infty w_{j+ iq}=\infty$ for $j=1,2,\dots,q$; and $\sum_{i=1}^\infty w_i^2<\infty$.

\item[A2] The map $F \mapsto \int p(x|\theta) f(\theta) d \theta$ is injective; that is the mixing density is identifiable from the mixture $\int p(x | \theta) f(\theta) d\theta.$ 

\item[A3] The map $\theta \to p(x|\theta)$ is bounded and continuous for $ x \in \chi $.

\item[A4] For $\theta_1$, $\theta_2$ and $\theta_3$ in $\Theta$, $\int \{\frac{p(x|\theta_1)}{p(x|\theta_2)}\}^k p(x|\theta_3)\nu(dx)<B$  for some $B>0$ and  for $k\leq 4q$.
\end{itemize}

 The conditions needed for convergence in $q$ dependent sequence is similar to that of independent case other than [A4] and [A1]. Condition [A1] is needed, as we will look at $q$ gap martingale sum, instead of the subsequence construction in Theorem \ref{wdependent}, \ref{wdependent2} for the gradually vanishing dependence. Condition [A4] is needed to account for the difference for the conditional mean terms, which will involve $q$ fold products of terms similar to $A_1(X_i)$'s. 
 Under the assumed conditions we have the following result. 
\begin{thm}
Let $X_1, \ldots, X_n$ be a sample form a $(q-1)$-dependent process with marginal density $m$ of the form \eqref{marginal} with mixing density $f$. Assume \rm{A1-A4} and \rm{C1-C3} hold and $\Theta$ is compact. Then  the estimator  $F_n$ from equation \eqref{newton} converges weakly to the true mixing distribution $F$ with probability one. 
\lb{mdependent}
\end{thm}
\begin{proof}
Given in the Appendix.
\end{proof}

Analogous to the general case, a statement can be made about convergence under misspecification of support in the $M$-dependent case. 
Assume the set up of Theorem~\ref{wdependent22}.

\begin{thm}[Convergence under miss-specification]
Assume the $X_i$ are generated from a $(q-1)$-dependent process with fixed marginal density $m$. Assume \rm{A1-A4} and \rm{C1-C3} hold. Then  $\tilde{K}_n^* \rightarrow 0$ with probability one where $\tilde{K}_n^*$ is defined in Theorem~\ref{wdependent22}. 
\lb{mdependent22}
\end{thm}
The proof is given in the Appendix.

\subsection{Convergence rate for $M$-dependent case}
Next we investigate the convergence rate of the PR algorithm in the $M$-dependent case.
We will assume $\rm{B1}'$ as the decay rate for the weights $w_i$. 
\begin{thm}
For  unique $\tilde{f}$, $w_i \sim i^{-\gamma}$ $\gamma \in (3/4,1)$, under the setting of Theorem \ref{mdependent22}, for $\gamma' \in (\gamma,1)$, $n^{-\gamma'+1}\tilde{K}_n^* \rightarrow 0$ with probability one and $H(m_{\tilde{f}},m_n)=o(n^{-(1-\gamma')/2})$,  if $\frac{m}{m_{\tilde{f}}}$ is bounded away from zero and infinity. 
\lb{rate_thm1}
\end{thm}

The proof of Theorem \ref{rate_thm1}  follows from   argument similar to that of    Theorem \ref{rate_thm2}  and given in the Appendix.

\section{Discussion}
We have established consistency  of the solution of the predictive recursion algorithm under various dependence scenarios. For the dependent cases considered, we have also explored convergence rate for the solution to the predictive recursion algorithm.  The  theoretical development provides justification for using the original algorithm in many cases.
 %We also provide an modified approach that is applicable to the  non stationary case. 
 Under stationarity but  misspecification   about dependence structure the original algorithm continues to work as long as the dependence decays reasonably fast.  Best possible nonparametric rate for dependent cases may be an interesting problem to explore and conditions for feasibility of minimax rate needs to be studied.

%We establish a natural  alternative  to the SA algorithm for the dependent cases with changing marginals  which  works under  conservative conditions. More general dependent scenarios or less conservative conditions should be investigated. It would be interesting to see to what extent condition AA4 can be relaxed.

The proposed theoretical development justifies the possible use of stochastic approximation even when we have error in observations coming from an moving average or autoregressive mean zero distribution, if certain conditions are satisfied. It is well known that stochastic approximation or predictive recursion algorithms do not give a posterior estimate. However,  similar  extension for posterior consistency under the misspecification of independence under conditions analogous to equation \eqref{strdep} may be explored.

\section{Acknowledgement}
 The first author Nilabja Guha is supported by the NSF grant \#2015460.

\setcounter{section}{0}
\renewcommand{\thesection}{\Alph{section}}

\section{ Appendix}
We first prove a simple proposition which we use throughout the proofs.  This is a standard result from probability theory, which we restate and prove  for convenience. 

\begin{proposition}
 Let $Z_1,Z_2,\dots $ be a sequence of random variables such that $\sum_{i=1}^\infty E[|Z_i|]<\infty$. Then $\sum_{i=1}^nZ_i$ converges almost surely and the limit is finite almost surely.
\lb{pos_sum}
\end{proposition}
\begin{proof}
Let $\Omega$ be the probability space corresponding to the joint distribution.  For some $\omega \in \Omega$, if $\sum |Z_i(\omega)| $converges then $\sum Z_i(\omega)$ converges.  Let $Z_\infty$ be the limit of $\sum |Z_i(\omega)| $ which is defined to be infinity at some $\omega$  incase $\sum |Z_i(\omega)|$ diverges to positive infinity.

By Monotone Convergence Theorem, $E[Z_\infty]=\sum_{i=1}^\infty E[|Z_i|]<\infty$ (or equivalently  can be argued using Fatou's lemma on the sequence of partial sums $\sum_{i=1}^n|Z_i|$). Hence $Z_\infty$ is finite with probability one. Therefore, $ \sum_{i=1}^\infty |Z_i|$ converges with probability one and hence $ \sum_{i=1}^\infty Z_i$ converges with probability one.
\end{proof}

\begin{lemma}
From equation \eqref{wkkl1}, $P(\sup_n | S_{v,j}^{(n)}|>\frac{\epsilon}{j^s})\leq C'_0\epsilon^{-2}j^{-K(2\alpha-1)+2s-1}$,  for $C'_0>0$,  some universal constant not depending on $j$.
 \lb{sup_lem}
\end{lemma}
\begin{proof}
From the derivation after equation \eqref{wkkl1}  and using Proposition  \ref{chaul_bd}, $E[v_{\Psi(i,j)}^2]\leq  E[2w_{\Psi(i,j)}^2(1+A_1(X_{\Psi(i,j)})^2]\leq c_0^{''}w^2_{\Psi(i,j)}$, for some universal constant $c_0^{''}>0$, using condition A4.  From the  martingale construction of the equation \ref{martcon} , $\Psi(1,j)>j^K$.
 By construction, the coefficients  belonging to each block ${\bf \tilde{B}}_l$ is less for the higher the index $j$ is, that is if $\Psi(l_1,j_1)$ and $\Psi(l_2,j_2)$ is in ${\bf \tilde{B}}_l$ then $w_{\Psi(l_1,j_1)}>w_{\Psi(l_2,j_2)}$ if $j_2>j_1$.  Hence, $E[(S_{v,j}^{(n)})^2]\leq j^{-1}c^{''}_0\sum_{l\geq\Psi(1,j)} w_l^2 \leq c^{''}_0j^{-1} \sum_{l\geq\Psi(1,j)} l^{-2\alpha}<c_0'j^{-1}{\Psi(1,j)}^{-2\alpha+1}<C_0'j^{-1}j^{-2K\alpha+K}$, where $c^{''}_0,c_0',C_0'$ are  universal constants.

%  From KL convergence, $KL(m, m_{n_k})\rightarrow KL(m,\hat{m})$ and also in $L1$ topology.  Hence, $\hat{m}={m}$ and from uniqueness of mixing distribution  $\hat{F}=F$ a,e. $\mu(\theta)$ and $F_n$ converges to $F$ weakly.

%\subsection{\bf  Proof of  Proposition \ref{infm_bound}}
%
%\begin{proof}
%By expanding equation \eqref{strdep} and using the fact $(x-1)\geq \log x$ the first part follows.
%
%The latter part follows from the fact that $c_1'<\frac{|\log(x)|}{|x-1|}<c_2'$ for some $c_1',c_2'>0$,  for $x$ bounded away from zero and infinity and  $x\neq 1$.
%
%\end{proof}

 Finally, using Doob's maximal inequality \citep{durrett2019probability},  \[P(\sup_n | S_{v,j}^{(n)}|>\frac{\epsilon}{j^s})=lim_{n_0\uparrow \infty}P(\sup_{n<n_0} | S_{v,j}^{(n)}|>\frac{\epsilon}{j^s}) \leq C'_0\epsilon^{-2}j^{-K(2\alpha-1)+2s-1}.\]
\end{proof}

\subsection{\bf  Proof of  Proposition \ref{chaul_bd}}
\begin{proof}
Note that $\frac{m_{f_1}(X_i)}{m_{f_2}(X_i)}\leq \frac{ p(X_i|\theta_1)}{p(X_i|\theta_2)}\leq A_1(X_i)$, where at $\theta_2$ and $\theta_1$, $p(X_i|\theta)$ is minimized and maximized, respectively when $X_i \notin \chi_c$ (note that $p(x|\theta)$ continuous function on a compact set). 

If $X_i\in \chi_c$ then $\frac{m_{f_1}(X_i)}{m_{f_2}(X_i)} \leq \frac{b }{a}\leq A_1(X_i)$. 

\end{proof}

\subsection{\bf Proof of Theorem \ref{wdependent2}}
\begin{proof}
Following \eqref{kl2} and the decomposition \eqref{wkkl1} in the dependent case,  we can have an analogous KL decomposition for the marginals in the dependent case. We have, 
\begin{eqnarray}
K_n^*-K_0^*=\sum_{j=1}^{{L_n}} S_{v,j}^{*,(n)}-\sum_{j=1}^{L_n} S_{M,j}^{*,(n)}+\sum_{i=1}^nS^{\Delta*}_i+\sum_{i=1}^n E^*_i . 
\lb{wkkl2}
\end{eqnarray}
Here
\begin{eqnarray}
 S_{v,j}^{*,(n)}&=&\sum_{i=1}^{N_{j,n}}v^*_{\Psi(i,j)}\nonumber\\
S_{M,j}^{*,(n)}&=&\sum_{i=1}^{N_{j,n}}w_{\Psi(i,j)}M^*_{\Psi(i,j)}\nonumber\\
v_{i}^*&=&w_i\Big((1-\int\frac {h_{i,X_i}(x)}{m_{i-1}(x)}m(x)\nu(dx))-E [1-\int_{\chi}  \frac{h_{i,X_i}(x)}{m_{i-1}(x)}m(x)\nu(dx) |\mathcal{F}_{last,i}]\Big)\nonumber\\
M^*_{i}&=&E [\int_{\chi}  \frac{h^{(i_{last})}_{i,X_i}(x)}{m_{i_{last}}(x)}m(x)\nu(dx) -1|\mathcal{F}_{last,i}];
h^{(i_{last})}_{i,x'}(x) \nonumber \\
&=&\frac{\int p(x'|\theta)p(x|\theta)f_{i_{last}}(\theta)d\mu( \theta)}{m_{i_{last}}(x')}\nonumber\\
S^{\Delta*}_i&=&w_iE [\int_{\chi}  \frac{h^{(i_{last})}_{i,X_i}(x)}{m_{i_{last}}(x)}m(x)\nu(dx) -1|\mathcal{F}_{last,i}] - \nonumber \\
&& w_iE [\int_{\chi}  \frac{h_{i,X_i}(x)}{m_{i-1}(x)}m(x)\nu(dx) -1|\mathcal{F}_{last,i}],\nonumber
 \lb{wkkl2_decomp}
\end{eqnarray}
where $h_{i,x'}(x)$ is defined in \eqref{kl2}.
From  the proof of Theorem \ref{wdependent} it follows that the sequence $K_n$ converges and $K_n^*$ converges to zero over a subsequence with probability one, as in that case $\Theta$ was finite. Therefore, we first show that $K_n^*$ converges to zero with probability one by showing that   $K_n^*$ converges almost surely.  

{ \underline{\it Convergence of $\sum_{j=1}^{{L_n}} S_{v,j}^{*,(n)}$:}} 
From earlier calculations, $|1-\int\frac {h_{i,X_i}(x)}{m_{i-1}(x)}m(x)\nu(dx)|\leq 1+A_1(X_i)$.  Hence, convergence of $\sum_{j=1}^{{L_n}} S_{v,j}^{*,(n)}$ follows exactly same as in the same way as  convergence of $\sum_{j=1}^{{L_n}} S_{v,j}^{(n)}$ in Theorem \ref{wdependent}.  As,  $S_{v,j}^{*,(n)}$ are squared integrable martingales each of them converge almost surely. 
From the fact,  that $|1-\int\frac {h_{i,X_i}(x)}{m_{i-1}(x)}m(x)\nu(dx)|\leq 1+A_1(X_i)$ we have $E[{v^*_{\Psi(i,j)}}^2]\leq  E[2w_{\Psi(i,j)}^2(1+A_1(X_{\Psi(i,j)}^2)]\leq c_0w^2_{\Psi(i,j)}$ for some fixed $c_0>0$. Therefore, argument similar to those in  Lemma \ref{sup_lem}, we have 
\[P(\sup_n | S_{v,j}^{(*,n)}|>\frac{\epsilon}{j^s})\leq c'_0 j^{-K(2\alpha-1)+2s-1},\] for $c'_0>0$.
Hence, only finitely many martingales make contribution to the tail sum with probability one as $\sum _j P(\sup_n | S_{v,j}^{*,n}|>\frac{\epsilon}{j^s})<\infty$ and $\sup_n | S_{v,j}^{*,n}|<\frac{\epsilon}{j^s}$ for all but finitely many $j$'s with probability one with $s>1$.  Thus,  $\sum_{j=1}^{{L_n}} S_{v,j}^{*,(n)}$ is Cauchy almost surely and therefore, convergent almost surely.

 \noindent \underline{\it Convergence of $\sum {S_i^{\Delta}}^*$:} 
 Let, 
\beqa
\Delta^*(X_i,x)&=& w_i|\frac{h_{i,X_i}^{(i_{last})}(x)}{m_{i_{last}}(x)}-\frac{h_{i,X_i}(x)}{m_{i-1}(x)}| \\
&\leq&  w_i\Huge\{\int_{\theta} |\frac{p(X_i|\theta)p(x|\theta)f_{i_{last}}(\theta)}{m_{i-1}(X_i)m_{i-1}(x)}(\frac{m_{i-1}(x)m_{i-1}(X_i)}{m_{i_{last}}(X_i)m_{i_{last}}(x)}-1)| d\mu (\theta)\\
&+&\int_{\theta}|\frac{p(X_i|\theta)p(x|\theta)}{m_{i-1}(X_i)m_{i-1}(x)}{f_{i_{last}}(\theta)}(\frac{f_{i-1}(\theta)}{f_{i_{last}}(\theta)}-1)|d\mu (\theta)\Huge\} 
\eeqa
and let $\int \Delta^*(X_i,x) m(x)\nu(dx)=\Delta^*(X_i)$. 
Note that 
\[\frac{m_{i-1}(x)}{m_{i_{last}}(x)}=\prod_{j=i_{last}}^{i-2}(1+w_j(\frac{\int p(X_j|\theta)p(x|\theta)f_j(\theta)d\theta}{m_{f_j}(x)m_{f_j}(X_j)}-1))\leq \prod_{j=i_{last}}^{i-2}(1+w_j(A_1(X_j)+1)) .\]
Hence, in the $l$ fold products of $A_1(X_{j_l})$'s where the index $j_l$ appears at most twice in $(\frac{m_{i-1}(x)m_{i-1}(X_i)}{m_{i_{last}}(X_i)m_{i_{last}}(x)}-1)$.
In $S^{\Delta*}_i$  the terms are  products of  terms like $\frac{p(X_i | \theta_{i_1})}{p(X_j | \theta_{i_2})}$ and $\frac{p(x | \theta_{j_1})}{p(x | \theta_{j_2})}$ where $\theta_{i_1}, \theta_{i_2}, \theta_{j_1}, \theta_{j_2}$ are in ${\Theta}_H$.    Let $q_i=i-i_{last}$.  Using  $\rm{B2}'$ and Holder's inequality, expectation of any $l$ fold such product would be bounded by $b_3^l$. The number of such products of $A_1(\cdot)$ is less than $q_i^l$, $q_i=o(i^{1/K+\epsilon})$ for any $\epsilon>0$.   Also, $l$ fold product contains $n_H^l$ many terms for which expectation can be bounded by $B_3^l$ for $B_3 > 0$. Hence, $E[|{S_i^\Delta} ^*|]\leq E|\Delta^*(X_i)|\leq  Cw_i\sum_{l=1}^\infty i^{-l(\alpha_1-1/K)}\tilde{B}_3^l<C_1i^{-2\alpha_1'}$ for some universal constants, $C,C_1$ and $\tilde{B}_3$ greater than zero, for $.5<\alpha_1'<\alpha_1<\alpha$ and large enough $K$. 
Thus, $\sum E[|{S_i^\Delta} ^*|] \preceq \sum i^{-2\alpha_1'}<\infty$, where for sequences $\{a_n\}_{n\geq 1}$, $\{b_n\}_{n\geq 1}$,$a_n,b_n>0$,  $a_n\preceq b_n$ implies that  $a_n\leq C_0'b_n$ for some $C_0'>0$. Hence,  $\sum {S_i^{\Delta}}^*$ converges with probability one. 

 \noindent \underline{\it Convergence of $\sum w_iM_i^*$:} 
 Note that \[M_i^*=\int \delta_i^*(x') m(x')dx'+\int \delta_i^*(x')(\frac{p_{i,last}}{(m(x)}-1)m(x')\nu(dx')\]\[=I+II\]
 where $\delta_i^*(x') = [\int_{\chi}  \frac{h^{(i_{last})}_{i,x'}(x)}{m_{i_{last}}(x)}m(x)\nu(dx) -1]$.  Using Cauchy-Schwartz inequality and  condition B3, the expectation of the second term is bounded by $C'\rho^{i-i_{last}}$ from some  $C'>0$. Number of times $(i-i_{last})$ is equal to $l$ is  less than $(l+1)^K$ where $K$ is defined in the martingale construction  \ref{martcon}.  Therefore, for $II$ the sum over all $i$ is absolutely convergent. 
 
The first term, 
\begin{eqnarray*}
I = \int_{\chi}   \delta_i^*(x') m(x')\nu(dx')&=& \int_{\chi}[\int_{\chi}\frac {h^{(i_{last})}_{i,x'}(x)}{m_{i_{last}}(x)}m(x)\nu(dx)\nu(dx')]-1 \nonumber \\
& =&E_{f_{i_{last}}}[\large(\int_{\chi}  \frac{p(x|\theta)}{m_{i_{last}}(x)} m(x)\nu(dx)\large)^2]-1  \nonumber\\
&\geq&  \large( E_{f_{i_{last}}}[\int  \frac{p(x|\theta)}{m_{i_{last}}(x)} m(x)\nu(dx)]\large)^2-1 \\
&=& 1-1=0.
 \end{eqnarray*}
Hence, $\sum w_iM_i^*$ either converges or diverges to $+\infty$. Given that LHS in equation \ref{wkkl2} cannot be $-\infty$ and the other terms in RHS in \ref{wkkl2} converges, $\sum w_iM_i^*$ has to converge with probability one. 
 
Hence, we have $K_n^*$ converging to zero in a sub-sequence with probability one, and converging with probability one. Hence, $K_n^*$ converges to zero with probability one.

 We now argue that this implies $F_n$ converges weakly to $F$. Suppose not. Since $\Theta$ is compact, $F_n$ is tight and hence have convergent subsequence. Let $F_{n_k}$ be a subsequence that converges to ${\hat{F}} \ne F.$ Let ${\hat{m}}$ be the marginal corresponding to ${\hat{F}}$. Then by [B4],  $m_{n_k}$ converges pointwise to ${\hat{m}}$ and hence by Scheffe's theorem it converges  in $L_1$ and hence in Hellinger norm to ${\hat{m}}$. 
However, by the previous calculations,  $m_{n_k}$  converges to $m$ almost surely in Kullback Leibler distance and therefore in Hellinger norm, which is a contradiction as $\hat{F}\neq F$. 
Hence, $F_n$ converges to $F$ weakly in every subsequence. 
 \end{proof}

\subsection{\bf  Proof of Theorem \ref{mdependent}} 
\lb{proof_mdependent}
\textcolor{white}{line}
\begin{proof}
For $ j \leq q$, define $l_j(q,n) =  \max \{l:  j + q(l-1) \leq n\}$.
Also let  for each $j$ the subsequences $\{X_{j + q(l-1)}, l = 1, \ldots, l_j(q,n) \}$ be denoted by $\{X_{j,l} \}.$ By construction $\{X_{j,l}, l = 1,\ldots, l_j(q,n)\}$ are iid with marginal distribution $m(\cdot)$. Let $\mathcal{F}_{j,l} = \sigma \langle X_{1}, \ldots, X_{j,l} \rangle$ denote  the $\sigma-$field generated by all $X_i$'s up to and including $X_{j,l}$. Also let the marginals $m_{j + q(l-1)}$ generated during the iterations be denoted by $m_{j,l}$ and the weights $w_{j + q(l-1)}$ be denoted by $w_{j,l}$. 
From equation \eqref{kl1} 
\begin{eqnarray}
K_n - K_0=\sum_{j=1}^{{q}} S_{j, n} - \sum_{j=1}^{q} Q_{j, n} + \sum_{i = 1}^{n } w_i D_{i}+\sum_{i=1}^n E_i .
\label{q_dep_decomp}
\end{eqnarray}
Here
\begin{eqnarray*}
 S_{j,n}&=& \sum_{l=1}^{l_j(q,n)}  w_{j,l} V_{j,l},\\
 V_{j,l} &=&(1  - \frac{m(X_{j,l})}{m_{j+q(l-1)-1}(X_{j,l}) }  ) - E\big[(1  -\frac{m(X_{j,l})}{m_{j+q(l-1)-1}(X_{j,l}) }  )|\mathcal{F}_{j,l-1}\big], \\
 M_{j,l} &=& E[(-1 + \frac{m(X_{j,l})}{m_{j,(l-1)}(X_{j,l}) }) | \mathcal{F}_{j,l-1}];  l> 1\text{ and }\\
 M_{j,l} &=&E\big[(\frac{m(X_{j,l})}{m_{j+q(l-1)-1}(X_{j,l}) } -1 )|\mathcal{F}_{j,l-1}]; l=1,\\
 Q_{j,n}&=&  \sum_{l=1}^{l_j(q,n)}  w_{j,l} M_{j,l},\\
D_{i} &=& E[\frac{m(X_i)}{m_{i-1}(X_i)}(\frac{m_{i-1}(X_i)}{m_{i-q}(X_i)} - 1)|\mathcal{F}_{j,l-1}];i>q \text{ and } D_i=0, \text{ for } i\leq q,\\
E_i &=& \int R(X_i, \theta) f(\theta) \dmu, \\
\end{eqnarray*}
where $R(X_i, \theta)$ is defined as in \eqref{KLdiff}.  We follow the same argument as in the general case; first we show that $K_n$ converges with probability one and $K_n^*$ converges  to zero over some subsequence with probability one. Then we show convergence of $K_n^*$.  As before,  we show convergence of $\sum_{j=1}^{{q}} S_{j, n}$, the remainder term $\Delta_n = \sum_{i = 1}^{n } w_iD_{i}$ and the error term $T_n = \sum_{i=1}^n E_i$. We simplify some of the expressions first.
From \eqref{ntnmrg}, for $i>q$
\[ \frac{m_{i-1}(X_i)}{m_{i-q}(X_i)} = \prod_{j=1}^{q-1}[ 1 + w_{i-j}(\frac{\int  \pxit p(X_{i-j} | \theta) f_{i-j-1}(\theta)  \dmu }{m_{i-j-1}(X_i) m_{i-j-1}(X_{i-j})}-1)] .\]
We have
\begin{eqnarray}
|(\frac{\int  \pxit p(X_{i-j} | \theta) f_{i-j-1}(\theta)  \dmu }{m_{i-j-1}(X_i) m_{i-j-1}(X_{i-j})}-1)|  \leq 1+
A_1(X_{i-j})   
\label{jensen}
\end{eqnarray}
by using triangle inequality and using Proposition \ref{chaul_bd} on 
$\frac{p(X_{i-j} | \theta)}{m_{i-j-1}(X_{i-j})}$.  
By Holder's inequality and assumption A4, we have 
\[ E[\prod_{j=1}^r  A_1^2(X_{i-j})] \leq (c_un_Hb/a)^{2q} B \]
for $r = 1, \ldots, q-1.$ Thus
\[ E[|1 - \frac{m_{i-1}(X_i)}{m_{i-q}(X_i)} |^2]  \leq w_{i-q}3^q {(c_un_Hb/a)^{2q}}B.\]
Similarly, we can bound $E[|\frac{m(X_i)}{m_{i-1}(X_i)}|^2] < (c_un_Hb/a)^{2}B.$
By Cauchy-Schwartz we have $E[|w_iD_i|] \leq C w_{i-q}^2$ where $C > 0$ is a constant. Hence by Proposition~\ref{pos_sum},  $\sum_{i=1}^n w_iD_i$ converges almost surely.  Therefore, $\Delta_n \rightarrow \Delta_\infty$, a finite random variable,  with probability one.
Note that $E[|E_i|]\leq 2w_i^2\int E\big[\big((\frac{p(X_i|\theta)}{m_{i-1}(X_i)}-1)^2\big]$ for $i$ greater than some positive integer $i_0$, as $w_i<1/2$ for large enough $i$.
Applying  A4, we have $E[|E_i|]<2(2(c_un_Hb/a)^2B+2)w_i^2$ and hence $\sum_{i>i_0'} E[|E_i|]$ converges and hence by Proposition~\ref{pos_sum}, $T_n$ converges almost surely. 
 Similarly, $E[V_{j,l}^2]\leq 2(1+n_H^2B)$ and $S_{j,n}$ is a martingale with filtration  $F_{j,l_j(q,n)}$ and $ E[S_{j,n}^2]\leq \sum 2w_{j,l}^2(1+n_H^2B)<\infty$ and hence, $S_{j,n}$ converges almost surely to a finite random variable as it is an square integrable martingale.

From equation \eqref{q_dep_decomp}  we have convergence of $\sum_{j=1}^q Q_{j,n} \to Q_{\infty} < \infty$ with probability one. This statement holds as  in L.H.S in \eqref{q_dep_decomp}   is a fixed quantity subtracted from a positive quantity and $\sum_{j=1}^{{q}} S_{j, n}$, $ \sum_{i = 1}^{n } w_i D_{i}$ and $\sum_{i=1}^n E_i$ converges with probability one.  As  $Q_{j,n}=  \sum_{l=1}^{l_j(q,n)}  w_{j,l} M_{j,l}$ and $M_{j,l}> \int log \frac{m_{}(x)}{m_{i-q}(x)}m(x)\nu(dx)=K_{{i-q}}^*\geq 0$, for $i>q$. Hence, any   $Q_{j,n}$ can not diverge to infinity. Moreover, as $\sum_{i=1}^\infty w_{j+ iq}=\infty$ for $j=1,2,\dots,q-1$ , $M_{j,l}$ has to go zero in some subsequence almost surely.

Next we show convergence of $K_n^*$.
Analogously, replacing $V_i$ and $M_i$ by $V_i^*$ and $M_i^*$  from the above derivation, from equation \eqref{kl2}  we get 
\begin{eqnarray}
K_n^* - K_0^*=\sum_{j=1}^{{q}} S^*_{j, n} - \sum_{j=1}^{q} Q^*_{j, n} + \sum_{i = 1}^{n } w_i D^*_{i}+\sum_{i=1}^n E^*_i 
\lb{mdepmarg}
\end{eqnarray}
where  $S_{j,n}^*$ the martingale sequences, which converges due to squared integrability.
\begin{eqnarray*}
 S^*_{j,n}&=& \sum_{l=1}^{l_j(q,n)}  w_{j,l} V^*_{j,l},\\
V_i^*&=&V^*_{j,l}=(1-\int\frac {h_{i,X_i}(x)}{m_{i-1}(x)}m(x)\nu(dx))-\\
&&E [1-\int_{\chi}  \frac{h_{i,X_i}(x)}{m_{i-1}(x)}m(x)\nu(dx) |\mathcal{F}_{j,l-1}];i=j+q(l-1),\\
M_i^*&=&E [\int_{\chi}  \frac{h^{(q)}_{i,X_i}(x)}{m_{j,l-1}(x)}m(x)\nu(dx) -1|\mathcal{F}_{j,l-1}], i>q\text{ and }\\
 M_i^*&=& E [\int_{\chi}  \frac{h_{i,X_i}(x)}{m_{i-1}(x)}m(x)\nu(dx) -1|\mathcal{F}_{j,l-1}];i\leq q, \text{ for } i=j+q(l-1) \text{ and } \\
h^{(q)}_{i,x'}(x)&=&\frac{\int p(x'|\theta)p(x|\theta)f_{i-q}(\theta)d\mu \theta}{m_{i-q}(x')},  \\
%\end{eqnarray*}
%\begin{eqnarray*}
Q_i^* & = &   \sum_{l=1}^{l_j(q,n)}  w_{j,l} M^*_{j,l}; M^*_{j,l}=M^*_i\text{ for } i=j+q(l-1),   \\  
D_i^* & = &E [\int_{\chi}  \frac{h^{(q)}_{i,x'}(x)}{m_{j,l-1}(x)}m(x)\nu(dx) -1|\mathcal{F}_{j,l-1}]-\\
  & &E [\int_{\chi}  \frac{h_{i,x'}(x)}{m_{i-1}(x)}m(x)\nu(dx) -1|\mathcal{F}_{j,l-1}]; i>q \text{ and } D^*_i=0;i\leq q,\\
R^*(X_i,x)&=&w_i^2(\frac{h_{i,X_i}(x)}{m_{i-1}(x)}-1)^2R(w_i[{h_{i,X_i}(x)}{m_{i-1}(x)}-1]),\\
 E_i^*&=&\int_{\chi} R(X_i,x)m(x)\nu(dx).
\end{eqnarray*}
Then for $i>q$,
\begin{eqnarray}
E[\int\frac {h_{i,X_i}^{(q)}(x)}{m_{i-q}(x)}m(x)\nu(dx)-1|F_{i-q}]= \int_{\chi} \int _{\chi}(\frac{h^{(q)}_{i,x'}(x)}{m_{i-q}(x)}m(x)m(x')\nu(dx)\nu(dx') -1) \nonumber \\
 =E_{f_{i-q}}[\large(\int  \frac{p(x|\theta)}{m_{i-q}(x)} m(x)\nu(dx)\large)^2]-1  \nonumber\\
\geq  \large( E_{f_{i-q}}[\int  \frac{p(x|\theta)}{m_{i-q}(x)} m(x)\nu(dx)]\large)^2-1 = 1-1=0.
 \end{eqnarray}
%As a result $S_{M^*,j}$ converges as $k_i\geq0$.
 Also $\sum_{i=1}^n E^*_i $ converges, using argument similar to that of  $\sum_{i=1}^nE_i$.  The proof of martingale squared integrability and the convergence of $S^*_{j,n}$ follows similarly as in $S_{j,n}$. The convergences of the difference term  $ \sum_{i = 1}^{n } w_i D^*_{i}$ follows similarly to the first part and given in the next subsection. Hence, we have $\sum_{j=1}^{q} Q^*_{j, n} $'s converging to finite quantities with probability one for each $j$,  as $K_n^*$ is positive. Hence, $K_n^*$ converges with probability one. 

From the fact $K_i^*=KL(m,m_i)$ converges almost surely as $i$ goes to infinity and it converges to zero in a subsequence almost surely, we have it converging to zero almost surely. Therefore, by arguments given in the proof of Theorem~\ref{wdependent2},   $F_n$ converges weakly to $F$  with probability one.
\end{proof}

\subsubsection{Convergence of $\sum_{i=1}^n E^*_i $ and $ \sum_{i = 1}^{n } w_i D^*_{i}$}

 \textcolor{white}{line}

The proof of martingale squared integrability and convergences of the reminder terms $\sum_{i=1}^n E^*_i $ and $ \sum_{i = 1}^{n } w_i D^*_{i}$ from Theorem~\ref{mdependent}:

\noindent \underline{\it Convergence of $S^*_{j,n}$:}
We show that $S^*_{j,n}$ is a square integrable martingale.  Note that $\frac{h_{i,X_i}(x)}{m_{i-1}(x)}\leq A_1(X_i)$.  From A4, $E[\big(\int \frac{h_{i,X_i}(x)}{m_{i-1}(x)}m(x)\nu(dx)\big)^2]<B$ using Holder's  inequality. Thus, 
we have $E\big[(S^*_{j,n})^2\big]\leq \sum 2w_i^2(B+2)<\infty$ which proves our claim.

 \noindent \underline{\it Convergence of $\sum E_i^*$:  }

Note that \[E[\sum_{i>i_0'}| E^*_i|]\preceq \sum w_i^2 E[(\int \frac{h_{i,X_i}(x)}{m_{i-1}(x)})^2 m(x)d\nu(x)+2]<\infty\]
from A4 and A1, and from the fact $w_i\downarrow 0$ and $w_i(\int_{\chi}  \frac{h^{(q)}_{i,X_i}(x)}{m_{j,l-1}(x)}m(x)\nu(dx) -1)>-w_i$ and    $E[A_1(X_i)^2]<\infty$ (following A4). Hence, $\sum_{i>i_0'}E_i^*$ and $\sum_{i}E_i^*$ converges with probability one.

\noindent \underline{\it Convergence of  $\sum w_i D_i^*$: }
Let,
\[ \Delta(X_i,x)=w_i|\frac{h_{i,X_i}^{(q)}(x)}{m_{i-q}(x)}-\frac{h_{i,X_i}(x)}{m_{i-1}(x)}|\leq w_i\Huge\{\int_{\theta}( |\frac{p(X_i|\theta)p(x|\theta)f_{i-q}(\theta)}{m_{i-1}(X_i)m_{i-1}(x)}(\frac{m_{i-1}(x)m_{i-1}(X_i)}{m_{i-q}(X_i)m_{i-q}(X_i)}-1)| d\mu (\theta)\] \[+\int_{\theta}|\frac{p(X_i|\theta)p(x|\theta)}{m_{i-1}(X_i)m_{i-1}(x)}{f_{i-q}(\theta)}(\frac{f_{i-1}(\theta)}{f_{i-q}(\theta)}-1)|)d\mu (\theta)\Huge\} \]
and $\int \Delta(X_i,x)m(x)\nu(dx)=\Delta(X_i)$.
From the fact 
\[ \frac{m_{i-1}(X_i)}{m_{i-q}(X_i)} = \prod_{j=1}^{q-1}[ 1 + w_{i-j}(\frac{\int  \pxit p(X_{i-j} | \theta) f_{i-j-1}(\theta)  \dmu }{m_{i-j-1}(X_i) m_{i-j-1}(X_{i-j})}-1)] \]
and we have 
\begin{eqnarray*}
|(\frac{\int  \pxit p(X_{i-j} | \theta) f_{i-j-1}(\theta)  \dmu }{m_{i-j-1}(X_i) m_{i-j-1}(X_{i-j})}-1)|  \leq 1+A_1(X_i).  \nonumber 
\end{eqnarray*} 
Similarly \[\frac{p(X_i|\theta)p(x|\theta)}{m_{i-1}(X_i)m_{i-1}(x)}\leq A_1(X_i) A_1(x).\]
The  part $|\frac{m_{i-1}(X_i)m_{i-1}(x)}{m_{i-q}(X_i)m_{i-q}(x)}-1|$ of R.H.S for  $\Delta(X_i,x)$ can be bounded by  the sum of $1\leq q' \leq 2q-2$ fold product of $ (1+A_1(X_i) )$  and $(1+A_1(x)) $'s multiplied by coefficient less than $w_{i-q}^{q'}$. Similarly. for $|\frac{f_{i-1}(\theta)}{f_{i-q}(\theta)}-1|$ we have  $1\leq q' \leq q-1$ fold products of $ 1+A_1(X_i) $'s.  Hence, integrating with respect to $m(x)$ and taking expectation we have a universal bound for any such product term from Holder's inequality. 
Hence, $E[\Delta(X_i)]\leq B_uw_{i-q}^2$ for $i>q$, for some universal constant $B_u>0$. 
As, $E\big(\sum_i |w_iD_i^*|\big)\preceq \sum_{i}w_{i}^2< \infty $, $\sum w_iD_i^*$ converges absolutely with probability one. Therefore, it converges with probability one.

\subsection{\bf Proof of Theorem \ref{mdependent22}}
\begin{proof}
From the derivation of equation \eqref{q_dep_decomp} using $\tilde{f}(\theta)$ instead of $f(\theta)$, i.e writing $\log\frac{f_i}{f_{i-1}}=\log\frac{f_i}{\tilde{f}}-\log\frac{f_{i-1}}{\tilde{f}}$, we have, for $\tilde{K}_n=KL(\tilde{f},f_n)$, 
\begin{eqnarray}
\tilde{K}_n - \tilde{K}_0=\sum_{j=1}^{{q}} \tilde{S}_{j, n} - \sum_{j=1}^{q} \tilde{Q}_{j, n} + \sum_{i = 1}^{n } w_i \tilde{D}_{i}+\sum_{i=1}^n \tilde{E}_i .
\label{q_dep_decomp22}
\end{eqnarray}
Here
\begin{eqnarray*}
 \tilde{S}_{j,n}&=& \sum_{l=1}^{l_j(q,n)}  w_{j,l} \tilde{V}_{j,l},\\
 \tilde{V}_{j,l} &=&(1  - \frac{\tilde{m}(X_{j,l})}{m_{j+q(l-1)-1}(X_{j,l}) }  ) - E\big((1  - \frac{\tilde{m}(X_{j,l})}{m_{j+q(l-1)-1}(X_{j,l}) }  )|\mathcal{F}_{j,l-1}\big), \\
 \tilde{M}_{j,l} &=& E((-1 + \frac{\tilde{m}(X_{j,l})}{m_{j,l-1}(X_{j,l}) }) | \mathcal{F}_{j,l-1}); l>1 \\
 \text{ and } \tilde{M}_{j,l} &=& E((-1 + \frac{\tilde{m}(X_{j,l})}{m_{j+q(l-1)-1}(X_{j,l}) }) | \mathcal{F}_{j,l-1})\text{ for }l=1\\
\tilde{Q}_{j,n}&=&  \sum_{l=1}^{l_j(q,n)}  w_{j,l} \tilde{M}_{j,l},\\
\tilde{D}_{i} &=& \frac{\tilde{m}(X_i)}{m_{i-1}(X_i)}(\frac{m_{i-1}(X_i)}{m_{i-q}(X_i)} - 1);i>q, \text { and } \tilde{D}_{i}=0  \text{ for } i\leq q,\\
\tilde{E}_i &=& \int R(X_i, \theta) \tilde{f}(\theta) \dmu, \\
\end{eqnarray*}
For $i>q$, $\tilde{M}_{j,l} \geq \int \log\frac{\tilde{m}(x)}{m_{j,l-1}(x)}m(x)\nu(dx)=KL(m,m_{j,l-1})-KL(m,\tilde{m})=KL(m,m_{j,l-1})-\tilde{k}>0$. Convergence of $\tilde{S}_{j,n}$, $\sum \tilde{D}_{i}$, $\sum \tilde{E}_{i}$ follows from the proof of \ref{mdependent}.  Thus, similarly  each $\tilde{Q}_{j,n}$ has to converge and hence, $KL(m,m_{j,l-1})-\tilde{k}$ converges to zero in some subsequence with probability one.  Convergence of $KL(m,m_i)$ follows from the proof of Theorem \ref{mdependent}, and completes the proof.

\end{proof}

\subsection{\bf Proof of Theorem \ref{wdependent22}}
\begin{proof}
From the proof of Theorem \ref{wdependent}, we decompose   $\log\frac{f_i}{f_{i-1}}=\log\frac{f_i}{\tilde{f}}-\log\frac{f_{i-1}}{\tilde{f}}$, and have an equation  analogous to  \eqref{wkkl1}
\begin{eqnarray}
\tilde{K}_n-\tilde{K}_0=\sum_{j=1}^{{L_n}} \tilde{S}_{v,j}^{(n)}-\sum_{j=1}^{L_n} \tilde{S}_{M,j}^{(n)}+\sum_{i=1}^n\tilde{S}^{\Delta}_i+\sum_{i=1}^n \tilde{E}_i 
\lb{wkkl11}
\end{eqnarray}
which is derived by using $\tilde{f}(\theta)$ instead of $f(\theta)$.
Analogous to equation \eqref{wdep_decmp} we have,
\begin{eqnarray*}
\tilde{M}_{\tilde{i}}= \int(\frac{\tilde{m}(x)}{{m}_{{i}_{last}}(x)}-1)m(x)\nu(dx)+\int(\frac{\tilde{m}(x)}{{m}_{{i}_{last}}(x)}-1).(\frac{p_{{i},last}(x)}{m(x)}-1)m(x)\nu(dx). 
\end{eqnarray*}
Using argument similar to Theorem \ref{mdependent22}, $KL(m,{m}_i)-KL(m,\tilde{m})$ goes to zero in some subsequence with probability one, as $j$ goes to infinity. Convergence of $KL(m,m_i)$ is essentially same proof as in Theorem~\ref{wdependent2}. Together they complete the proof. 
\end{proof}

\subsection{\bf Proof of Theorem \ref{rate_thm1}}

\begin{proof}

Since $\frac{w_{i-1}}{w_i}=O(1)$, from the proof of $\ref{mdependent}$ we can conclude that  $c_1' \sum w_{i}\tilde{K}^*_{i-q}\leq\sum_i w_{i-q+1}\tilde{K}^*_{i-q}\leq c' \sum w_{i}\tilde{K}^*_{i-q}<\infty$ with $c'>0,c'_1>0, i=q+jl, l=1,2,\dots$ converges  with probability one.

Let $a_n=\sum_{i=1}^nw_i$, $a_0=0$. Hence, for $i=j+ql$ for  Theorem \ref{mdependent22}
\begin{eqnarray}
a_i\tilde{K}_i^*=a_{i-1}\tilde{K}_{i-1}^*+w_i\tilde{K}^*_{i-1}+a_iw_iV_{j,l}^*-a_iw_iM^*_{j,l}+a_iw_iD_i^*+a_iE_i^*\nonumber \\
\label{ratemart1}
\end{eqnarray}
and 
\begin{eqnarray}
a_n\tilde{K}_n^*-a_0\tilde{K}^*_0=\sum w_i\tilde{K}^*_{i-1}+\sum_{j=1}^q \sum_{l=1}^{l_j(q,n)} a_iw_iV_{j,l}^*-\sum_{j=1}^q \sum_{l=1}^{l_j(q,n)}a_iw_iM^*_{j,l}+\nonumber \\ \sum_{i=1}^n a_iw_iD_i^*+\sum_{i=1}^n a_iE_i^*.
\label{ratemart2}
\end{eqnarray}
Let $w_i=i^{-\alpha}, \alpha \in (3/4,1)$ and $a_i \sim i^{-(\alpha-1)}$,  $w_i'=a_iw_i\sim i^{-(2\alpha-1)}.$
For $M$ dependent case, from the proof of Theorem \ref{mdependent}, writing   $w_i'$ instead of $w_i$  we can  show the almost sure convergence $S^*_{j,l}$. Convergence of $\sum a_iw_iD_i^*$ and $\sum_ia_iE_i^*$ also follow  in similar fashion. From the fact $\frac{w_{i-1}}{w_i}=O(1)$ the term $\sum_iw_i\tilde{K}^*_{i-1}$ converges with probability one. Hence  following  A1, we have $ \sum_{j=1}^q \sum_{l=1}^{l_j(q,n)}a_iw_iM^*_{j,l}$ converging and   $M^*_{j,l}$ going to zero in a subsequence  with probability one for all $j$ as $l$ goes to infinity, and $a_n\tilde{K}_n^*$ converging with probability one. 

We have,   $a_n\tilde{K}^*_n$ converging to a finite number  for each $\omega$ outside a set of probability zero, and $a_n \sim n^{1-\alpha}$, then $n^{1-\gamma}\tilde{K}_n^*$ goes to zero with probability one for $\gamma>\alpha$.    The conclusion about the Hellinger distance follows from the relation between KL and Hellinger distance from argument as in Theorem \ref{rate_thm2} proof. 

\end{proof}

\bibliography{References}

\end{document}